\documentclass[12pt,reqno]{amsart}
\usepackage[utf8]{inputenc}
\usepackage[T1]{fontenc}
\usepackage{a4wide}
\usepackage{amsmath}
\usepackage{amscd}
\usepackage{amssymb,latexsym}
\usepackage[stable]{footmisc}
\usepackage{xcolor}
\usepackage[colorlinks,linkcolor=teal,citecolor=teal]{hyperref}

\numberwithin{equation}{section}
\newtheorem{cor}{\bf Corollary}[section]
\newtheorem{lemma}{\bf Lemma}[section]
\newtheorem{prop}{\bf Proposition}[section]
\newtheorem{thm}{\bf Theorem}[section]
\newtheorem{definition}{Definition}
\newtheorem{remark}{Remark}

\begin{document}
	\title{Formally Integrable  Structures I. Resolution of Solution Sheaves}
	\author{Qingchun Ji}
	\address{School of Mathematical Sciences, Fudan University, and Shanghai Center for Mathematical Sciences, Shanghai 200433, China}
	\email{qingchunji@fudan.edu.cn}
	\author{Jun Yao}
	\address{School of Mathematical Sciences, Fudan University, Shanghai 200433, China}
	\email{jyao\_fd@fudan.edu.cn}
	\author{Guangsheng Yu}
	\address{School of Mathematics and Statistics, Ningbo University, Ningbo 315211, China}
	\email{yuguangsheng@nbu.edu.cn}
	\thanks{This work was partially supported by National Natural Science Foundation of China, No. 12431005, No. 12271275 and Ningbo Youth Science and Technology Innovation Leading Talent Project, No. 2023QL014.}
	\subjclass[2020]{Primary 35A01; Secondary 35N10, 35A20, 35A27}
	\keywords{Resolution, Cousin type problems, local solvability, the Treves complex, Levi flat structures}
	\date{}

\begin{abstract}
	 This is the first of a series of papers on the $L^2$-theory for formally integrable structures. It is devoted to constructing a resolution of the  solution sheaf for a class of overdetermined systems introduced by L. H{\"o}rmander. A sufficient condition for global exactness is obtained, which leads to gluing techniques for local solutions formulated as Cousin type problems. In addition, we also prove the local solvability of the Treves complex for formally integrable structures with vanishing Levi forms, including Levi flat structures as special cases. To the best of the authors' knowledge, nothing more than the elliptic case is known about the local $L^2$-solvability of the Treves complex in the Levi flat case.
\end{abstract}
\maketitle

\section{Introduction}
Given a manifold $X$, various extensions of the classical Frobenius theorem and Newlander-Nirenberg theorem concern the local integrability of involutive  structures over $X$, i.e., formally integrable complex subbundles  of the complexified tangent bundle $\mathbb{C}TX$. The first significant attempt in this direction was made by L. Nirenberg(\cite{N}), a complex Frobenius theorem was established for arbitrary complex subbundle $E\subseteq\mathbb{C}TX$ under the assumptions that $E+\bar{E}$ is also a subbundle of $\mathbb{C}TX$ and that both $E$ and $E+\bar{E}$ are formally integrable, i.e.,
\begin{equation}
	\left[\Gamma(U,E),\Gamma(U,E)\right]\subseteq\Gamma(U,E),\left[\Gamma(U,E),\Gamma(U,\bar{E})\right]\subseteq\Gamma(U,E+\bar{E})\label{bracket}
\end{equation}
for any open subset $U\subseteq X$, where $\bar{E}$ is the complex-conjugate of $E$ and $[\cdot,\cdot]$ denotes the Lie bracket of vector fields. Such a subbundle $E$ is called a Levi flat structure.

Since the first exposition of the Levi flat structures(\cite{N}), there have been considerable developments and numerous new applications in this subject(see \cite{H1}, \cite{Web}, \cite{W}, \cite{HT}, \cite{G}, \cite{S} and references therein). Inspired by Levi flat structures, L. H{\"o}rmander introduced a class of first order overdetermined systems of partial differential equations(\cite{H1}) and established existence theorems. In \cite{Siu}, Y.-T. Siu pioneered the development of the multiplier ideal sheaf technique for general systems of partial differential equations. For $ x=(x_1,\cdots,x_n)\in \mathbb{R}^n $, we denote $ \partial/\partial x_\nu $ by $ \partial_\nu $ for $ 1\leq \nu\leq n $. To extend Nirenberg's complex Frobenius theorem, H{\"o}rmander(\cite{H1}) considered, under involutivity conditions motivated by $(\ref{bracket})$, the following general overdetermined system for the unknown $u$
\begin{equation}
	P_ju=f_j, 1\leq j\leq r,\label{OS}
\end{equation} 
where $P_j$'s are first order differential operators in an open subset $ \Omega\subseteq\mathbb{R}^n $ given by
\begin{equation}\label{operator}
	P_j=\sum_{\nu=1}^{n}a_j^\nu\partial_\nu+a_j^0,
\end{equation}
where $a_j^\nu\in C^{\infty}\left(\Omega\right)$ for $0\leq \nu\leq n,1\leq j \leq r$.



We will investigate the system \eqref{OS} from the perspective of the sheaf $\mathcal{S}_P$ of germs of locally square-integrable solutions of the homogeneous system $Pu=0$, where $P:=(P_1,\cdots,P_r)$ is the $r$-tuple of differential operators given by (\ref{operator}). 
Let $\mathcal{L}^{\oplus\tbinom{r}{q}}$ be a sheaf over $\Omega$ defined by (\ref{cp.}) for $0\leq q\leq r$, it's by definition a fine sheaf of Abelian groups. Section \ref{s2} is devoted to  constructing a differential complex over $\Omega$ 
\begin{equation}
	\mathcal{L}\stackrel{\mathcal{P}_1}{\longrightarrow}\mathcal{L}^{\oplus r}\stackrel{\mathcal{P}_2}{\longrightarrow}\mathcal{L}^{\oplus\tbinom{r}{2}}\stackrel{\mathcal{P}_3}{\longrightarrow}\mathcal{L}^{\oplus\tbinom{r}{3}}\stackrel{\mathcal{P}_4}{\longrightarrow}\cdots\stackrel{\mathcal{P}_r}{\longrightarrow}\mathcal{L}^{\oplus\tbinom{r}{r}}\longrightarrow 0,\label{resolution}
\end{equation} 
such that 
\begin{equation}
	\mathcal{P}_1= P,\label{P_1}
\end{equation} 
where $ \tbinom{r}{q}=r!/q!(r-q)! $ is the binomial coefficient. In this way, we eventually get a resolution of $\mathcal{S}_P$. Moreover, since $\mathcal{L}^{\oplus\tbinom{r}{q}}$ is a fine sheaf, the above resolution also implies isomorphisms for $1\leq q\leq r$
\begin{equation}
	H^q(\Omega,\mathcal{S}_P)\cong \frac{{\rm{Ker}}\left(L_{loc}^2\left(\Omega\right)^{\oplus\tbinom{r}{q}}\stackrel{\mathcal{P}_{q+1}}{\longrightarrow}L_{loc}^2\left(\Omega\right)^{\oplus\tbinom{r}{q+1}}\right)}{{\rm{Im}}\left(L_{loc}^2\left(\Omega\right)^{\oplus\tbinom{r}{q-1}}\stackrel{\mathcal{P}_{q}}{\longrightarrow}L_{loc}^2\left(\Omega\right)^{\oplus\tbinom{r}{q}}\right)},\label{ISO}
\end{equation}
where $L_{loc}^2\left(\Omega\right)$ is the space of complex valued functions on $\Omega$ which are locally square-integrable. We will consider global exactness of the previous resolution, i.e., the vanishing of $H^q(\Omega,\mathcal{S}_P)$ in Section \ref{s3}, as well as the local $ L^2 $-solvability and the smooth solvability in the sense of germs of the Treves complex, which follows as a consequence(Corollary \ref{lst}). There have been many deep developments on the local solvability of the Treves complex(see \cite{CH}, \cite{BCH}, \cite{T1,T1.5,T2} and references therein). In Section \ref{s4}, we will formulate additive and multiplicative Cousin problems for $P=(P_1,\cdots,P_r)$, and in this way provide a sheaf-theoretic approach to gluing local solutions of the overdetermined system (\ref{OS}) by solving Cousin type problems(Theorem \ref{Cousin}).

This paper aims to lay a foundation for the $L^2$-technique for general overdetermined systems of PDEs arising from the Levi flat structures, so we leave alone the invariance problem for the complex \eqref{resolution} and this has been sufficient for the applications to the local solvability of Treves complex and gluing techniques for overdetermined systems of PDEs. We will introduce the notion of canonical bundle for a formally integrable structure, which is a natural generalization of that of a complex structure, to handle the invariance problem for the complex \eqref{resolution} in a subsequent paper.

We adopt the summation convention over repeated indices. For a multi-index $ J=(j_1,\cdots,j_q) $ where $1\leq j_1, j_2,\cdots, j_q\leq r$, we call $q$ the length of $J$(denoted by $|J|=q$). A multi-index $J$ is called strictly increasing if $1\leq j_1<j_2<\cdots<j_q\leq r$. Without special explanation, summation is always taken over strictly increasing multi-indices. 
For a strictly increasing multi-index $ J $, $ s\in\lbrace1,\cdots,r\rbrace,\,j\in J $ and $ 
k\in\lbrace1,\cdots,r\rbrace\setminus J $, we introduce the following notations
$$J\setminus j:={\text{the strictly  increasing multi-index with components in}}\ J \setminus\lbrace j\rbrace, $$
$$Jk:= {\text{the strictly  increasing multi-index with components in}}\ J\cup\lbrace k\rbrace,$$
$$(k,J):= {\text{the number of components in}}\ J \ {\text{which are less than}}\ k, $$
\begin{equation*} 
	sJ:=(s,j_1,\cdots,j_q),\ \text{and}\ (j,sJ\setminus j):= \left\{\begin{array}{rcl}(j,J\setminus j)+1,  & \text{if} \ j\ne s, \\ 0, & \text{if} \ j=s.\end{array}\right.
\end{equation*}
For $1\leq q\leq r$, we write $\xi\in\mathbb{C}^{\tbinom{r}{q}}$ as $\xi=(\xi_J)_{|J|=q}$ where $J$ runs over multi-indices of length $q$ and components $\xi_J\in \mathbb{C}$ are anti-symmetric in $j_1,j_2,\cdots,j_q$. It is clear that $\xi\in\mathbb{C}^{\tbinom{r}{q}}$ is uniquely determined by these $\xi_J$ with strictly increasing $J$.  We will denote the $L^2$-inner product and the usual inner product of the complex Euclidean spaces by $(\cdot,\cdot)$ and $\langle\cdot,\cdot\rangle$ respectively.

\section{Construction of the differential complex}\label{s2}
When an $r$-tuple $P=(P_1,\cdots,P_r)$  of first order partial differential operators is normalized and in the normal Cartan form,  there is an explicit procedure for constructing resolutions of the solution sheaf $\mathcal{S}_P$ by using compatibility operators(\cite{T}). It is known that the non-existence of a compatibility operator is a consequence of the strong degeneracy of $P$. 

To construct the desired complex (\ref{resolution}), we make the following assumption on the considered differential operators $P_1,\cdots,P_r$. 

$\mathbf{(A1)}$ For all $ i,j,k\in\lbrace1,\cdots,r\rbrace $, \begin{equation}\label{integrability condition}
	\left[P_j,P_k\right]:=P_jP_k-P_kP_j=c_{jk}^lP_l \ \ \text{and} \ \
	c_{jk}^i=-c_{kj}^i,
\end{equation} 
where  $ c_{jk}^1,\cdots,c_{jk}^r\in C^{\infty}\left(\Omega\right) $ are complex valued functions. This is analogous to the first condition in (\ref{bracket}). The next assumption will play an important role in our construction of compatibility  operators.

$\mathbf{(A2)}$ For $1\leq k,l,k',l'\leq r$, 
\begin{equation}
    c_{kk'}^sc_{sl'}^l+c_{k'l'}^sc_{sk}^l+c_{l'k}^sc_{sk'}^l=p_{l'}c_{kk'}^l+p_kc_{k'l'}^l+p_{k'}c_{l'k}^l,\label{J}  
\end{equation} 
where $ p_j $ is the principal part of $ P_j $, i.e.,
\begin{equation}
	p_j=a_j^\nu(x)\partial_\nu.\label{p_j}
\end{equation}

\begin{remark}
	We have the following sufficient condition for $ (\ref{J}) $
\begin{equation}
	{\rm rank}_{\mathbb{C}}\left(a_j^\nu(x)\right)_{0\leq \nu\leq n\atop1\leq j\leq r}=r,\label{J'}
\end{equation} 
which holds for all $x\in\Omega$. By the following identity
\begin{align*}
	\left[P_k,\left[P_{k'},P_{l'}\right]\right]+\left[P_{k'},\left[P_{l'},P_{k}\right]\right]+\left[P_{l'},\left[P_k,P_{k'}\right]\right]&=P_kP_{k'}P_{l'}-P_kP_{l'}P_{k'}-P_{k'}P_{l'}P_k+P_{l'}P_{k'}P_k\\
	&+P_{k'}P_{l'}P_{k}-P_{k'}P_{k}P_{l'}-P_{l'}P_{k}P_{k'}+P_{k}P_{l'}P_{k'}\\
	&+P_{l'}P_kP_{k'}-P_{l'}P_{k'}P_k-P_kP_{k'}P_{l'}+P_{k'}P_kP_{l'}\\
	&=0,
\end{align*}
it follows that
\begin{align*}
	\left(c_{kk'}^sc_{sl'}^l-p_{l'}c_{kk'}^l+c_{k'l'}^sc_{sk}^l-p_kc_{k'l'}^l-c_{kl'}^sc_{sk'}^l+p_{k'}c_{kl'}^l\right)P_l\equiv0.
\end{align*}
Therefore, the condition $ (\ref{J}) $ follows from $ (\ref{J'}) $.
\end{remark}

Now, we use $ P_1,\cdots,P_r $ and assumptions (A1)-(A2) to construct a sequence of differential operators
\begin{equation}\label{complex}
	L_{loc}^2(U)\stackrel{\mathcal{P}_1}{\longrightarrow}L_{loc}^2(U)^{\oplus r}\stackrel{\mathcal{P}_2}{\longrightarrow}L_{loc}^2(U)^{\oplus\tbinom{r}{2}}\stackrel{\mathcal{P}_3}{\longrightarrow}\cdots\stackrel{\mathcal{P}_r}{\longrightarrow}L_{loc}^2(U)^{\oplus\tbinom{r}{r}}
\end{equation}
for every open subset $U\subseteq \Omega$ where each $\mathcal{P}_q :L_{loc}^2(U)^{\oplus\tbinom{r}{q-1}}\rightarrow L_{loc}^2(U)^{\oplus\tbinom{r}{q}}$ is a densely defined operator to be determined($1\leq q\leq r$).

For $ q=1 $, as indicated by (\ref{P_1}), we define $ \mathcal{P}_1u=\left(P_1u,\cdots,P_ru\right) $ for $u\in L_{loc}^2(U)$. To find $\mathcal{P}_2$ such that $\mathcal{P}_2\circ\mathcal{P}_1=0$, let's first consider $f=(f_j)_{1\leq j\leq r}$ given by $f_j=P_ju$ for some $u\in L_{loc}^2(U), 1\leq j\leq r$. By (\ref{integrability condition}), we have $$P_if_j=[P_i,P_j]u+P_jP_iu=c_{ij}^sf_s+P_jf_i,$$ which indicates the following compatibility operator $\mathcal{P}_2$ for $\mathcal{P}_1$,
$$ \mathcal{P}_2f=\left(P_if_j-P_jf_i-c_{ij}^sf_s\right)_{1\leq i,j\leq r}, \  f=(f_j)_{1\leq j\leq r}\in L_{loc}^2(U)^{\oplus r}, $$ 
where $ c_{mn}^s$'s  are functions appearing in ($ \ref{integrability condition}$), then it is easy to see $\mathcal{P}_2\circ\mathcal{P}_1=0$.

In general, for $ 1\leq q\leq r $, we define $ \mathcal{P}_q $ as follows
\begin{equation}\label{boundary operator}
	\mathcal{P}_qf=\bigg(\sum_{j\in J}(-1)^{(j,J\setminus j)}P_jf_{J\setminus j}
    -\sum_{m<n\atop m,n\in J}{\rm sgn}\tbinom{J}{mnJ\setminus\lbrace m,n\rbrace}
	c_{mn}^sf_{sJ\setminus\lbrace m,n\rbrace}\bigg)_{|J|=q},
\end{equation}
where $ f=(f_I)_{|I|=q-1}\in 
L_{loc}^2(U)^{\oplus\tbinom{r}{q-1}}$. In (\ref{boundary operator}), $ \mathcal{P}_q 
$ is understood as the maximal extension of its restriction to smooth vector-valued functions, and thereby a densely defined operator between Fr{\'e}chet spaces $L_{loc}^2(U)^{\oplus\tbinom{r}{q-1}}$ and $L_{loc}^2(U)^{\oplus\tbinom{r}{q}}$.

The sequence of differential operators in (\ref{complex}) gives rise to the following sequence of sheaf morphisms over $\Omega$,
\begin{equation}\label{reso}
	\mathcal{L}\stackrel{\mathcal{P}_1}{\longrightarrow}\mathcal{L}^{\oplus r}\stackrel{\mathcal{P}_2}{\longrightarrow}\mathcal{L}^{\oplus\tbinom{r}{2}}\stackrel{\mathcal{P}_3}{\longrightarrow}\mathcal{L}^{\oplus\tbinom{r}{3}}\stackrel{\mathcal{P}_4}{\longrightarrow}\cdots\stackrel{\mathcal{P}_r}{\longrightarrow}\mathcal{L}^{\oplus\tbinom{r}{r}},
\end{equation}
where $\mathcal{L}^{\oplus\tbinom{r}{q}}$ for $0\leq q\leq r$ is the sheaf generated by the following complete presheaf
\begin{align}\label{cp.}
	\mathcal{L}^{\oplus\tbinom{r}{q}}(U):=\left\{f\in L_{loc}^2(U)^{\oplus\tbinom{r}{q}}\ |\ \mathcal{P}_{q+1}f\in L_{loc}^2(U)^{\oplus\tbinom{r}{q+1}}\right\},
\end{align}
where $U$ is any open subset of $\Omega$. The complex (\ref{reso}) restricts to 
\begin{equation}
	\mathcal{A}\stackrel{\mathcal{P}_1}{\longrightarrow}\mathcal{A}^{\oplus r}\stackrel{\mathcal{P}_2}{\longrightarrow}\mathcal{A}^{\oplus\tbinom{r}{2}}\stackrel{\mathcal{P}_3}{\longrightarrow}\mathcal{A}^{\oplus\tbinom{r}{3}}\stackrel{\mathcal{P}_4}{\longrightarrow}\cdots\stackrel{\mathcal{P}_r}{\longrightarrow}\mathcal{A}^{\oplus\tbinom{r}{r}},\label{resolutionS}
\end{equation} 
where $\mathcal{A}$ is the sheaf of germs of smooth functions.

\begin{remark}
	If $ r=n $, and $ 
	P_j=\partial_j $ for $ 
	j\in\lbrace1,\cdots,n\rbrace $, then $ (\ref{resolutionS}) $ is the de Rham complex.
	If $ n $ is even, $ r=n/2 $, and $ 
	P_j=\partial_{\bar{z}_j} $ for $ j\in\lbrace1,\cdots,r\rbrace $, then 
	the complex $ (\ref{resolutionS}) $ is the 
	Dolbeault complex.
\end{remark}

First, we need to show that the sequence ($ \ref{resolution} $) is actually a complex, i.e., $ \mathcal{P}_{q+1}\circ\mathcal{P}_{q}=0 $.

\begin{prop}
	$ \mathcal{P}_{q+1}\circ\mathcal{P}_{q}=0 $ for $ 1\leq q\leq r$ where $\mathcal{P}_{r+1}:=0$.
\end{prop}

\begin{proof}
    Let $U\subseteq\Omega$ be an open subset and $ f=(f_I)_{|I|=q-1}\in 
	L_{loc}^2(U)^{\oplus\tbinom{r}{q-1}}$, it follows 
	from the definition (\ref{boundary operator}) that
	\begin{align}
		\mathcal{P}_{q+1}\circ\mathcal{P}_{q}f=&\bigg(\sum_{k\in K}(-1)^{(k,K\setminus k)}P_k(\mathcal{P}_qf)_{K\setminus k}-\!\!\!\sum_{m<n\atop m,n\in K}{\rm sgn}\tbinom{K}{mnK\setminus\lbrace m,n\rbrace}
		c_{mn}^s(\mathcal{P}_qf)_{sK\setminus\lbrace m,n\rbrace}\bigg)_{|K|=q+1}\notag\\
		=&\bigg(\sum_{k\in K}\sum_{j\in K\setminus k}(-1)^{(k,K\setminus 
		k)+(j,K\setminus\lbrace k,j\rbrace)}P_kP_jf_{K\setminus\lbrace 
		k,j\rbrace}\notag\\
    	&-\sum_{m<n\atop m,n\in K} {\rm sgn}\tbinom{K}{mnK\setminus\lbrace m,n\rbrace}c_{mn}^s\sum_{j\in sK\setminus\lbrace m,n\rbrace}(-1)^{(j,sK\setminus\lbrace m,n,j\rbrace)}P_jf_{sK\setminus\lbrace m,n,j\rbrace}\notag\\
    	&-\sum_{k\in K}(-1)^{(k,K\setminus k)}\sum_{m'<n'\atop m',n'\in K\setminus k}{\rm sgn}\tbinom{K\setminus k}{m'n'K\setminus\lbrace k,m',n'\rbrace}P_k(c_{m'n'}^lf_{lK\setminus\lbrace k,m',n'\rbrace})\notag\\   	
    	&+\sum_{m<n\atop m,n\in K}\sum_{m''<n''\atop m'',n''\in sK\setminus\lbrace m,n\rbrace}{\rm sgn}\tbinom{K}{mnK\setminus\lbrace m,n\rbrace}c_{mn}^s\notag\\
    	&\quad\times{\rm sgn}\tbinom{sK\setminus\lbrace 
    	m,n\rbrace}{m''n''sK\setminus\lbrace 
    	m,n,m'',n''\rbrace}c_{m''n''}^lf_{lsK\setminus\lbrace 
    	m,n,m'',n''\rbrace}\bigg)_{|K|=q+1}\notag\\
    	=:& I-II-III+IV\label{I-IV}.
    \end{align}
    We will handle these four terms $I\sim IV$ separately as follows.     
    \begin{align*}
    	I=&\sum_{k\in K}\sum_{j\in K\setminus k}\left((-1)^{(k,K\setminus k)+(j,K\setminus\lbrace j,k\rbrace)}P_kP_jf_{K\setminus\lbrace j,k\rbrace}\right)_{|K|=q+1}\nonumber\\
    	=&\sum_{k<j\atop j,k\in K}\left((-1)^{(k,K\setminus k)+(j,K\setminus\lbrace j,k\rbrace)}(P_kP_j-P_jP_k)f_{K\setminus\lbrace j,k\rbrace}\right)_{|K|=q+1}\nonumber\\
    	=&\sum_{m<n\atop m,n\in K}\left({\rm sgn}\tbinom{K}{mnK\setminus\lbrace m,n\rbrace}\left[P_m,P_n\right]f_{K\setminus\lbrace m,n\rbrace}\right)_{|K|=q+1}\nonumber\\
    	=&\sum_{m<n\atop m,n\in K}\left({\rm sgn}\tbinom{K}{mnK\setminus\lbrace m,n\rbrace}c_{mn}^sP_sf_{K\setminus\lbrace m,n\rbrace}\right)_{|K|=q+1}.
    \end{align*}
	The second line is valid since $ (k,K\setminus k)+(j,K\setminus\lbrace 
	j,k\rbrace)+1=(j,K\setminus j)+(k,K\setminus\lbrace j,k\rbrace) $, if $ 
	k<j$ and $j,k\in K $. The last line is an application of ($ 
	\ref{integrability condition} $). We rewrite $II$ as
	\begin{align*}
		II=&\bigg(\sum_{m<n\atop m,n\in K} {\rm sgn}\tbinom{K}{mnK\setminus\lbrace m,n\rbrace}c_{mn}^sP_sf_{K\setminus\lbrace m,n\rbrace}\bigg)_{|K|=q+1}\\
		&+\bigg(\sum_{m<n\atop m,n\in K} {\rm sgn}\tbinom{K}{mnK\setminus\lbrace m,n\rbrace}c_{mn}^l
		\sum_{j\in K\setminus\lbrace m,n\rbrace}(-1)^{(j,lK\setminus\lbrace m,n,j\rbrace)}P_jf_{lK\setminus\lbrace m,n,j\rbrace}\bigg)_{|K|=q+1}\\ =:&  II'+II'',
	\end{align*}
	which gives
	\begin{equation}
		I=II' \label{I-II}.
	\end{equation}
	Similarly, we decompose $III$ into two parts
	\begin{align*}
		III=&\bigg(\sum_{k\in K}(-1)^{(k,K\setminus k)}\sum_{m'<n'\atop m',n'\in K\setminus k}{\rm sgn}\tbinom{K\setminus k}{m'n'K\setminus\lbrace k,m',n'\rbrace}
		c_{m'n'}^lP_kf_{lK\setminus\lbrace k,m',n'\rbrace}\bigg)_{|K|=q+1}\notag\\
		&+\bigg(\sum_{k\in K}(-1)^{(k,K\setminus k)}\sum_{m'<n'\atop m',n'\in K\setminus 
		k}{\rm sgn}\tbinom{K\setminus k}{m'n'K\setminus\lbrace 
		k,m',n'\rbrace}p_k(c_{m'n'}^l)f_{lK\setminus\lbrace k,m',n'\rbrace}\bigg)_{|K|=q+1}\notag\\  =: &III'+III''\label{III}
	\end{align*}
	where we denote by $ p_j $ the principal part of $ P_j $. Since
    $$ \sum_{k\in K}\sum_{m'<n'\atop m',n'\in K\setminus k}c_{m'n'}^lP_kf_{lK\setminus\lbrace k,m',n'\rbrace}=\sum_{m<n\atop m,n\in K}\sum_{j\in K\setminus\lbrace m,n\rbrace}c_{mn}^lP_jf_{lK\setminus\lbrace j,m,n\rbrace},$$
    $II''$ has the same number of terms as $III'$. From     
    \begin{align*}
    	\tbinom{K}{mnK\setminus\lbrace k,m,n\rbrace}=&\tbinom{K}{kK\setminus 
    	k}\tbinom{kK\setminus k}{mnkK\setminus\lbrace 
    	k,m,n\rbrace}\tbinom{mnkK\setminus\lbrace 
    	k,m,n\rbrace}{mnK\setminus\lbrace m,n\rbrace}\\
        =&\tbinom{K}{kK\setminus k}\tbinom{kK\setminus k}{kmnK\setminus\lbrace 
        k,m,n\rbrace}\tbinom{kK\setminus\lbrace k,m,n\rbrace}{K\setminus\lbrace 
        m,n\rbrace}\\
        =&\tbinom{K}{kK\setminus k}\tbinom{K\setminus k}{mnK\setminus\lbrace 
        k,m,n\rbrace}\tbinom{kK\setminus\lbrace k,m,n\rbrace}{K\setminus\lbrace 
        m,n\rbrace},        
    \end{align*}
    and 
    \begin{equation*}
    	(-1)^{(k,lK\setminus\lbrace k,m,n\rbrace)}
    	=
    	{\rm sgn}\tbinom{lK\setminus\lbrace m,n\rbrace}{klK\setminus\lbrace 
    	k,m,n\rbrace}
    	=
    	-{\rm sgn}\tbinom{K\setminus\lbrace m,n\rbrace}{kK\setminus\lbrace 
    	k,m,n\rbrace},
    \end{equation*}
    it follows that
    $$ (-1)^{(k,K\setminus k)}{\rm sgn}\tbinom{K\setminus 
    k}{mnK\setminus\lbrace k,m,n\rbrace}
    =-(-1)^{(k,lK\setminus\lbrace m,n,k\rbrace)}{\rm 
    sgn}\tbinom{K}{mnK\setminus\lbrace m,n\rbrace}, $$
    which indicates that corresponding terms in $II''$ and $III'$ have opposite signs, and therefore 
    \begin{equation}
    	II''+III'=0 \label{II-III}.
    \end{equation}  
    For the last term $IV$, we have
    \begin{align*}
    	IV=&\bigg(\sum_{m<n\atop m,n\in K}\sum_{m''<n''\atop m'',n''\in K\setminus\lbrace 
    	m,n\rbrace}{\rm sgn}\tbinom{sK\setminus\lbrace 
    	m,n\rbrace}{m''n''sK\setminus\lbrace m,n,m'',n''\rbrace}\nonumber\\
    	&\qquad\qquad\qquad\qquad\qquad \times{\rm sgn}\tbinom{K}{mnK\setminus\lbrace 
    	m,n\rbrace}c_{mn}^sc_{m''n''}^lf_{lsK\setminus\lbrace 
    	m,n,m'',n''\rbrace}\bigg)_{|K|=q+1}\label{IV(1)}\notag\\
    	&+\bigg(\sum_{m<n\atop m,n\in K}\sum_{m''<n''\atop m'',n''\in sK\setminus\lbrace 
    	m,n\rbrace}{\rm sgn}\tbinom{sK\setminus\lbrace 
    	m,n\rbrace}{m''n''sK\setminus\lbrace m,n,m'',n''\rbrace}\nonumber\\
    	&\qquad\qquad\qquad\qquad\qquad \times{\rm sgn}\tbinom{K}{mnK\setminus\lbrace 
    	m,n\rbrace}c_{mn}^sc_{m''n''}^lf_{lsK\setminus\lbrace 
    	m,n,m'',n''\rbrace}\bigg)_{|K|=q+1}\notag\\ =: & IV'+IV''.
    \end{align*}
    It's obvious that 
    \begin{equation} 
    	IV'\label{IV'}=0,
    \end{equation} 
    indeed,
    
    \begin{align*}
    	IV'&=\bigg(\sum_{m<n\atop m,n\in K}\sum_{m''<n''\atop m'',n''\in K\setminus\lbrace 
    	m,n\rbrace}{\rm sgn}\tbinom{K}{mnK\setminus\lbrace 
    	m,n\rbrace}\tbinom{sK\setminus\lbrace 
    	m,n\rbrace}{m''n''sK\setminus\lbrace m,n,m'',n''\rbrace}\\
    	&\qquad\quad \times 
    	c_{mn}^sc_{m''n''}^l(f_{lsK\setminus\lbrace 
    	m,n,m'',n''\rbrace}+f_{slK\setminus\lbrace 
    	m,n,m'',n''\rbrace})\bigg)_{|K|=q+1}=0.
    \end{align*}

    Now, it remains to handle terms $ III'' $ and $ 
    IV'' $. For a given multi-index $ K $,  each 
    $ f_{lK\setminus\lbrace k,m',n'\rbrace} $  in $ III'' $ is exactly given by one of the following expressions(without loss of generality, assume $ 
    k<m'<n' $)
    \begin{equation}
    	\left\{
    	\begin{aligned}
    		&(-1)^{(n',K\setminus n')}{\rm sgn}\tbinom{K\setminus n'}{km'K\setminus\lbrace k,m',n'\rbrace}p_{n'}(c_{km'}^l)f_{lK\setminus\lbrace k,m',n'\rbrace},\\
    	    &(-1)^{(k,K\setminus k)}{\rm sgn}\tbinom{K\setminus k}{m'n'K\setminus\lbrace k,m',n'\rbrace}
    	    p_k(c_{m'n'}^l)f_{lK\setminus\lbrace k,m',n'\rbrace},\\
    		&(-1)^{(m',K\setminus m')}{\rm sgn}\tbinom{K\setminus m'}{kn'K\setminus\lbrace k,m',n'\rbrace}p_{m'}(c_{kn'}^l)f_{lK\setminus\lbrace k,m',n'\rbrace}.
    	\end{aligned}
        \right.\label{III''}
    \end{equation}
    The corresponding terms in $ IV''$ are given by 
    \begin{equation}
    	\left\{
    	\begin{aligned}
    		{\rm sgn}\tbinom{K}{km'K\setminus\lbrace 
    		k,m'\rbrace}{\rm sgn}\tbinom{sK\setminus\lbrace 
    		k,m'\rbrace}{sn'K\setminus\lbrace 
    		k,m',n'\rbrace}c_{km'}^sc_{sn'}^lf_{lK\setminus\lbrace 
    		k,m',n'\rbrace},\\
    		{\rm sgn}\tbinom{K}{m'n'K\setminus\lbrace 
    		m',n'\rbrace}{\rm sgn}\tbinom{sK\setminus\lbrace 
    		m',n'\rbrace}{skK\setminus\lbrace k,m',n'\rbrace} 
    		c_{m'n'}^sc_{sk}^lf_{lK\setminus\lbrace k,m',n'\rbrace},\\
    		{\rm sgn}\tbinom{K}{kn'K\setminus\lbrace 
    		k,n'\rbrace}{\rm sgn}\tbinom{sK\setminus\lbrace 
    		k,n'\rbrace}{sm'K\setminus\lbrace k,m',n'\rbrace} 
    		c_{kn'}^sc_{sm'}^lf_{lK\setminus\lbrace k,m',n'\rbrace}.
    	\end{aligned}
        \right.\label{IV''}
    \end{equation}
    
    From 
    \begin{align*}
    	\tbinom{K}{km'K\setminus\lbrace k,m'\rbrace}=&\tbinom{K}{n'K\setminus 
    	n'}\tbinom{n'K\setminus n'}{km'n'K\setminus\lbrace 
    	k,m',n'\rbrace}\tbinom{km'n'K\setminus\lbrace 
    	k,m',n'\rbrace}{km'K\setminus\lbrace k,m'\rbrace}\\
    	=&\tbinom{K}{n'K\setminus n'}\tbinom{n'K\setminus 
    	n'}{n'km'K\setminus\lbrace k,m',n'\rbrace}\tbinom{n'K\setminus\lbrace 
    	k,m',n'\rbrace}{K\setminus\lbrace k,m'\rbrace}\\
    	=&\tbinom{K}{n'K\setminus n'}\tbinom{K\setminus 
    	n'}{km'K\setminus\lbrace k,m',n'\rbrace}\tbinom{n'K\setminus\lbrace 
    	k,m',n'\rbrace}{K\setminus\lbrace k,m'\rbrace},      
    \end{align*}
    we know
    \begin{align}
    	{\rm sgn}\tbinom{K}{km'K\setminus\lbrace k,m'\rbrace}{\rm sgn}\tbinom{sK\setminus\lbrace 
    	k,m'\rbrace}{sn'K\setminus\lbrace k,m',n'\rbrace}
    	=&{\rm sgn}\tbinom{K}{n'K\setminus n'}{\rm sgn}\tbinom{K\setminus 
    	n'}{km'K\setminus\lbrace k,m',n'\rbrace}\notag\\ =& (-1)^{(n',K\setminus n')}{\rm sgn}\tbinom{K\setminus n'}{km'K\setminus\lbrace 
    	k,m',n'\rbrace}.\label{III-IV'}
    \end{align}
    Similarly, we have the following equalities for the other two parts
    \begin{equation}
    	{\rm sgn}\tbinom{K}{m'n'K\setminus\lbrace m',n'\rbrace}\tbinom{sK\setminus\lbrace 
    	m',n'\rbrace}{skK\setminus\lbrace k,m',n'\rbrace} =(-1)^{(k,K\setminus k)}{\rm sgn}\tbinom{K\setminus k}{m'n'K\setminus\lbrace k,m',n'\rbrace},\label{III-IV''}
    \end{equation}
    \begin{equation}
    	{\rm sgn}\tbinom{K}{kn'K\setminus\lbrace 
    	k,n'\rbrace}\tbinom{sK\setminus\lbrace 
    	k,n'\rbrace}{sm'K\setminus\lbrace k,m',n'\rbrace} =(-1)^{(m',K\setminus m')}{\rm sgn}\tbinom{K\setminus m'}{kn'K\setminus\lbrace k,m',n'\rbrace}.\label{III-IV'''}
    \end{equation}
    In the same manner, it is easy to see
    \begin{align}
    	{\rm sgn}\tbinom{K}{km'K\setminus\lbrace 
    	k,m'\rbrace}\tbinom{sK\setminus\lbrace 
    	k,m'\rbrace}{sn'K\setminus\lbrace k,m',n'\rbrace}=& {\rm 
    	sgn}\tbinom{K}{m'n'K\setminus\lbrace 
    	m',n'\rbrace}\tbinom{sK\setminus\lbrace 
    	m',n'\rbrace}{skK\setminus\lbrace k,m',n'\rbrace}\notag \\ =&-{\rm sgn}\tbinom{K}{kn'K\setminus\lbrace 
    	k,n'\rbrace}\tbinom{sK\setminus\lbrace 
    	k,n'\rbrace}{sm'K\setminus\lbrace k,m',n'\rbrace}.\label{III-IV''''}
    \end{align}
    From (\ref{III''})$\sim$(\ref{III-IV''''}) and the assumption (A2), it follows that
    \begin{equation}
    	III''=IV''.\label{L}
    \end{equation}
    Now, $ \mathcal{P}_{q+1}\circ\mathcal{P}_{q}f=0 $ is a consequence of (\ref{I-IV})$\sim$(\ref{IV'}) and (\ref{L}).		
\end{proof} 
\begin{remark} 
	It seems that the cases of $q\geq 2$ are different from  the case $q=1$, our proof depends on the assumption ${\rm (A2)}$ for $q\geq 2$. \end{remark}

We denote by $ ^{t}\mathcal{P}_q $ the formal adjoint of the differential operator $\mathcal{P}_q $. For later use, we need to compute the expression of $ ^{t}\mathcal{P}_q f$ in terms of derivatives of its components. For $ \xi, \eta\in\mathbb{C}^{\tbinom{r}{q}}$, let $ \langle \xi,\eta\rangle:=\sum_{|J|=q}\langle \xi_J,\eta_J\rangle$ where $\langle\cdot,\cdot\rangle$ on the right-hand side is the standard Hermitian inner product on $\mathbb{C}$. Given  $ f,g\in L_{loc}^2\left(\Omega\right)^{\oplus\tbinom{r}{q}}$ one of which has compact support, we denote
$ (f,g):=\int_{\Omega}\langle f,g\rangle=\sum_{|J|=q}(f_J,g_J) $. The formal adjoint $ ^{t}\mathcal{P}_q $ of $ \mathcal{P}_q $ is then determined by $ 
(\mathcal{P}_qg,f)=(g,{}^{t}\mathcal{P}_qf) $ for all  $(f,g)\in L_{loc}^2\left(\Omega\right)^{\oplus\tbinom{r}{q}}\times\mathcal{D}\left(\Omega\right)^{\oplus\tbinom{r}{q-1}}$ where $\mathcal{D}\left(\Omega\right)$ is the space of smooth functions with compact support in $\Omega$. We have the following formula of $^{t}\mathcal{P}_q$.

\begin{prop}
	For $ 1\leq q\leq r $ and $ f=(f_J)_{|J|=q}\in 
	L_{loc}^2\left(\Omega\right)^{\oplus\tbinom{r}{q}}$, 
	\begin{equation}\label{formal adjoint of boundary operator}
		^{t}\mathcal{P}_qf=\bigg(\sum_{j\notin I}{}^{t}P_jf_{jI}
		-\sum_{s\in I}\sum_{m<n\atop m,n\notin I\setminus s}(-1)^{(s,I\setminus s)}{\rm sgn}\tbinom{(I\setminus s)mn}{mnI\setminus s}\bar{c}_{mn}^sf_{(I\setminus{s})mn}\bigg)_{|I|=q-1},
	\end{equation}
	where $ ^{t}P_j $ is the formal adjoint of $ P_j (1\leq j\leq r)$.	
\end{prop}

\begin{proof}
	Let $ g=(g_I)_{|I|=q-1}\in \mathcal{D}\left(\Omega\right)^{\oplus\tbinom{r}{q-1}}$, then we have
	\begin{align*}
		&(\mathcal{P}_qg,f)\\
		&=\sum_{|J|=q}\bigg(\sum_{j\in J}(-1)^{(j,J\setminus j)}P_jg_{J\setminus j}-\sum_{m<n\atop m,n\in J}{\rm sgn}\tbinom{J}{mnJ\setminus \lbrace m,n\rbrace}c_{mn}^sg_{sJ\setminus \lbrace m,n\rbrace}, f_J\bigg)\\
		&=\sum_{|J|=q}\bigg[\sum_{j\in J}(-1)^{(j,J\setminus j)}\left(P_jg_{J\setminus j},f_J\right)-\sum_{m<n\atop m,n\in J}{\rm sgn}\tbinom{J}{mnJ\setminus \lbrace m,n\rbrace}\left(c_{mn}^sg_{sJ\setminus \lbrace m,n\rbrace}, f_J\right)\bigg]\\
		&=\sum_{|I|=q-1}\sum_{j\notin I}(-1)^{(j,I)}\left(P_jg_{I},f_{Ij}\right)-\sum_{|I'|=q-2}\sum_{m<n\atop m,n\notin I'}{\rm sgn}\tbinom{I'mn}{mnI'}\left(c_{mn}^sg_{sI'}, f_{I'mn}\right)\\
		&=\sum_{|I|=q-1}\bigg[\sum_{j\notin I}(-1)^{(j,I)}\left(g_{I},{}^tP_jf_{Ij}\right)-\sum_{s\in I}\sum_{m<n\atop m,n\notin I\setminus s}{\rm sgn}\tbinom{(I\setminus s)mn}{mn(I\setminus s)}(-1)^{(s,I\setminus s)}\left(c_{mn}^sg_I, f_{(I\setminus s)mn}\right)\bigg]\\
		&=\sum_{|I|=q-1}\bigg(g_{I},\sum_{j\notin I}{}^tP_jf_{jI}-\sum_{s\in I}\sum_{m<n\atop m,n\notin I\setminus s}{\rm sgn}\tbinom{(I\setminus s)mn}{mn(I\setminus s)}(-1)^{(s,I\setminus s)}\bar{c}_{mn}^sf_{(I\setminus s)mn}\bigg)\\
		&=(g,{}^{t}\mathcal{P}_qf).
	\end{align*}
The third line follows from setting $ I=J\setminus j $ and $ I'=J\setminus\lbrace m,n\rbrace $. The fourth line is valid by letting $ I=I's $.
\end{proof}
\begin{remark}
	When $r\geq2$, we have the following explicit formulae of the formal adjoint operators $^{t}\mathcal{P}_q$ $(q=1,2)$:
	$^{t}\mathcal{P}_1f=\sum_{j=1}^{r}{}^{t}P_jf_{j}$ where $f=(f_j)_{1\leq j\leq r}\in L_{loc}^2\left(\Omega\right)^{\oplus r}$, and $^{t}\mathcal{P}_2f=\big(\sum_{j\neq i}{}^{t}P_jf_{ji}+\sum_{m<n} \bar{c}_{mn}^if_{mn}\big)_{1\leq i\leq r}$ where $f=(f_{ij})_{1\leq i<j\leq r}\in 
L_{loc}^2\left(\Omega\right)^{\oplus\tbinom{r}{2}}$.
\end{remark}

We conclude this section by relating (\ref{resolutionS}) to the Treves complex. Let $X$ be a manifold, and $E\subseteq \mathbb{C}TX$ be a formally integrable subbundle. For $q\geq 1$ and open subset $\Omega\subseteq X$, we denote by $$\mathfrak{N}_E^q\left(\Omega\right)=\left\{\omega\in A^q\left(\Omega\right) \ | \ \omega(X_1,\cdots,X_q)=0 \ \text{for} \ X_1,\cdots, X_q\in\Gamma(\Omega,E)\right\},$$ where $A^q\left(\Omega\right)$ is the space of complex smooth $q$-forms on $\Omega$. By the formal integrablity of $E$, it follows from Cartan's magic formula
\begin{align}\label{property of exterior diff.}
	{\rm 
	d}\omega\left(X_1,\cdots,X_{q+1}\right)=&\sum_{j=1}^{q+1}(-1)^{j+1}X_j\left(\omega\left(X_1,\cdots,\hat{X}_j\cdots,X_{q+1}\right)\right)\nonumber\\
	&+\sum_{j<k}(-1)^{j+k}\omega\left(\left[X_j,X_k\right],X_1,\cdots,\hat{X}_j,\cdots,\hat{X}_k,\cdots,X_{q+1}\right),
\end{align}
that $\mathfrak{N}_E^{\bullet}\left(\Omega\right)$ is preserved by the exterior derivative, i.e.,
\begin{equation*}
	{\rm d}\mathfrak{N}_E^q\left(\Omega\right)\subseteq \mathfrak{N}_E^{q+1}\left(\Omega\right).
\end{equation*} 
Treves(\cite{T1}) introduced the following complex as the quotient complex of the de Rahm complex
\begin{equation}\label{new complex}
	C^\infty\left(\Omega\right)\stackrel{{\rm d}'}{\longrightarrow}\mathfrak{U}_E^1\left(\Omega\right)\stackrel{{\rm d}'}{\longrightarrow}\mathfrak{U}_E^2\left(\Omega\right)\stackrel{{\rm d}'}{\longrightarrow}\mathfrak{U}_E^3\left(\Omega\right)\stackrel{{\rm d}'}{\longrightarrow}\cdots,
\end{equation}
where $$\mathfrak{U}_E^q\left(\Omega\right):=\frac{A^q\left(\Omega\right)}{\mathfrak{N}_E^q\left(\Omega\right)}, \ q\geq 1.$$

We choose a frame field $\lbrace P_1,\cdots,P_r, L_1,\cdots,L_m\rbrace $ of $\mathbb{C}TX$ over a coordinate chart $\Omega$, such that vector fields $P_1,\cdots,P_r$ (i.e., first order differential operators without zero order terms) span $E$ over $\Omega$, and let $ \lbrace 
\omega_1,\cdots,\omega_r,\theta_1,\cdots,\theta_m\rbrace $ be the dual frame of $\mathbb{C}T^*X$. Given $ f=\sum_{|I|=q-1}f_I\omega_I\in A^{q-1}\left(\Omega\right)$,
\begin{equation}
	{\rm d}f \equiv \sum_{|J|=q}\bigg(\sum_{j\in J}(-1)^{(j,J\setminus j)}P_jf_{J\setminus j}\bigg) \omega_J+\sum_{|I|=q-1}f_I {\rm d} \omega_I \ 
	{\rm mod} \ \mathfrak{N}_E^{q}\left(\Omega\right).\label{d}
\end{equation}
By (\ref{property of exterior diff.}), we also have 
\begin{align*}
	{\rm d} \omega_I\left(P_{i_1},\cdots,P_{i_q}\right)=&\sum_{j=1}^{q+1}(-1)^{j+1}P_{i_j}\left(\omega_I\left(P_{i_1},\cdots,\hat{P}_{i_j}\cdots,P_{i_{q+1}}\right)\right)\\
	&+\sum_{j<k}(-1)^{j+k}\omega_I\left(\left[P_{i_j},P_{i_k}\right],P_{i_1},\cdots,\hat{P}_{i_j},\cdots,\hat{P}_{i_k},\cdots,P_{i_q}\right)\\
	=&\sum_{j<k}(-1)^{j+k}c_{i_ji_k}^s\omega_I\left(P_s,P_{i_1},\cdots,\hat{P}_{i_j},\cdots,\hat{P}_{i_k},\cdots,P_{i_q}\right),
\end{align*}
which, in conjunction with (\ref{d}), implies
\begin{equation}
	{\rm d}'\bigg(f \ {\rm mod} \ \mathfrak{N}_E^{q-1}\left(\Omega\right)\bigg)=\sum_{|J|=q}\bigg(\mathcal{P}_qF\bigg)_J\omega_J \ {\rm mod} \ \mathfrak{N}_E^{q}\left(\Omega\right),\label{T1}
\end{equation}
where $F:=\big(f_I\big)_{|I|=q-1}\in C^\infty\left(\Omega\right)^{\oplus\tbinom{r}{q-1}}$. 
In particular, 
\begin{equation}
	{\rm d}'\bigg(f \ {\rm mod} \ \mathfrak{N}_E^{q-1}\left(\Omega\right)\bigg)=0 \ {\rm if \ and \ only \ if} \  \mathcal{P}_qF=0.\label{T2}
\end{equation}
Hence, the Treves complex (\ref{new complex}) can be locally realized as our previous complex (\ref{resolutionS}) where $P_1,\cdots,P_r$ are complex smooth vector fields(i.e., without zero order terms), and this enables us to consider the local solvability of the Treves complex(Corollary \ref{lst}) by $L^2$-theory of the complex (\ref{resolution}).

\section {Global exactness of the differential complex}\label{s3}
In this section, we will provide a criterion for the global exactness of (\ref{resolution}), which consequently leads to its local exactness. Following \cite{H1}, we make a further assumption on the differential operator $P=(P_1,\cdots,P_r)$.

$\mathbf{(A3)}$ There are functions $ d_{jk}^l $ and $ e_{jk}^l $ in $ C^{\infty}\left(\Omega\right) $ for 
all $ j,k\in\lbrace1,\cdots,r\rbrace $ such that 
\begin{equation}\label{condition of the Lie braket}
	\left[p_j,\bar{p}_k\right]=d_{jk}^lp_l-e_{jk}^l{}\bar{p}_l,
\end{equation}
where $p_j$ is the principal part of $P_j$ defined by (\ref{p_j}) and $ \bar{p}_j $ is the operator obtained by conjugating $ p_j $'s coefficients, i.e.,
\begin{equation*}
	\bar{p}_j=\bar{a}_j^\nu(x)\partial_\nu \ {\rm for} \ 1\leq j\leq r.
\end{equation*}

\begin{remark}
	The coefficients $ d_{jk}^l $ and $ e_{jk}^l $ are 
	not uniquely determined.
\end{remark}		

By using $p_j,\bar{p}_j(1\leq j\leq r)$, we define the following quadratic form for every $C^2$-function on $\Omega$ which was originally introduced by H{\"o}rmander(\cite{H2}).

\begin{definition}\label{Quadratic Form}
	Let $ \varphi\in C^2\left(\Omega\right) $ be a real valued function, $x\in \Omega$, the quadratic form $Q_{\varphi,x}$ on $ \mathbb{C}^{r}$ is defined as
	\begin{equation}
        Q_{\varphi,x}(\xi,\xi):={\rm Re}\left(p_j\bar{p}_k\varphi(x)+e_{jk}^l\bar{p}_l\varphi(x)\right)\xi_{j}\bar{\xi}_{k}	\label{quadratic form}
    \end{equation}
    for all $ \xi=(\xi_j)_{1\leq j\leq r}\in\mathbb{C}^{r} $. 
\end{definition}
\begin{remark}
$(1)$ Since $\varphi$ is real valued, it follows from $(\ref{condition of the Lie braket})$ that for every  $ \xi=(\xi_j)_{1\leq j\leq r}\in\mathbb{C}^{r}$
\begin{align*}
	0={\rm Re}\left[{\xi}_jp_j,\bar{\xi}_k\bar{p}_k\right]\varphi(x)={\rm Re}\:d_{jk}^lp_l\varphi(x)\xi_{j}\bar{\xi}_{k}-{\rm Re}\:e_{jk}^l\bar{p}_l\varphi(x)\xi_{j}\bar{\xi}_{k},
\end{align*}
which shows that substituting $e_{jk}^l\bar{p}_l$ with $d_{jk}^lp_l$ in Definition \ref{Quadratic Form} yields the same quadratic form $Q_{\varphi,x}$. \\
$(2)$ Let $E\subseteq \mathbb{C}T\Omega$ denote the subbundle spanned by the principal parts $p_1,\cdots,p_r$, the quadratic form $Q_{\varphi,x}$ depends only on $E$ if $E+\bar{E}$ also forms a subbundle$($i.e., has constant rank$)$. In this sense, $Q_{\varphi,x}$ can be viewed as an analogue of the Hessian of the Bott connection for a foliation structure.
\end{remark}
For each $1\leq q\leq r$, $Q_{\varphi,x}$ induces a quadratic form $Q_{q,\varphi,x}$ on $ \mathbb{C}^{\binom{r}{q}}$
by 
\begin{equation*}
	Q_{q,\varphi,x}(\xi,\xi):=\sum_{|I|=q-1}Q_{\varphi,x}(\xi_{(I)},\xi_{(I)})
\end{equation*} 
for all $\xi=(\xi_J)_{|J|=q}\in\mathbb{C}^{\binom{r}{q}}$, where $\xi_{(I)}:=(\xi_{1I},\cdots,\xi_{rI})^T\in\mathbb{C}^r$ for any multi-index $I$ of length $q-1$.

\begin{lemma}\label{relationship of eigenvalues}
	For a real valued function $ \varphi\in C^2\left(\Omega\right) $, $x\in \Omega, 1\leq q\leq r$, denote by $ \lbrace\lambda_j\rbrace_{j=1}^r $ the eigenvalues of the quadratic form $Q_{\varphi,x}$$($w.r.t. the standard inner product on $\mathbb{C}^r$$ ) $, then the eigenvalues$ ( $w.r.t. the standard inner product on $\mathbb{C}^{\binom{r}{q}}$$ ) $ of the quadratic form $ Q_{q,\varphi,x} $ are exactly given by $ \lambda_J:=\sum_{j\in J}\lambda_j $, indexed by  strictly increasing multi-indices $J$ of length $q$.
\end{lemma}

\begin{proof}
	Fix some $x\in\Omega$, find an $r$ by $r$ unitary matrix $A=(a_{ij})_{1\leq i,j\leq r}$ which diagonalizes $Q_{\varphi,x}$, i.e.,
	\begin{equation}
	Q_{\varphi,x}(A\xi,A\xi)=\sum_{1\leq j\leq r}\lambda_j|\xi_j|^2,  \ {\rm for} \ \xi=(\xi_j)_{1\leq j \leq r}\in\mathbb{C}^r.\label{diag}
    \end{equation}
	Let $A_q$ be the unitary transform on $\mathbb{C}^{\tbinom{r}{q}}$ which is induced by $A$ as follows
	\begin{align*}
		A_q:\ \mathbb{C}^{\tbinom{r}{q}}&\longrightarrow \mathbb{C}^{\tbinom{r}{q}}\\
		(\xi)_{j_1\cdots j_q}&\mapsto(A_q\xi)_{k_1\cdots k_q}:=\sum_{1\leq j_1,\cdots,j_q\leq r}a_{k_1j_1}\cdots a_{k_qj_q}\xi_{j_1\cdots j_q}
	\end{align*}
	for all $\xi=(\xi_J)_{|J|=q}\in\mathbb{C}^{\tbinom{r}{q}}$. Set $\eta=A_q\xi\in\mathbb{C}^{\tbinom{r}{q}}$, then for every multi-index $I=(i_1,\cdots,i_{q-1})$
	\begin{equation}
		\eta_{(I)}=\sum_{1\leq l_1,\cdots,l_{q-1}\leq r}a_{i_1l_1}\cdots a_{i_{q-1}l_{q-1}}A\xi_{(L)}\in\mathbb{C}^r,\label{U}
	\end{equation}
    or equivalently 
	\begin{equation}
		A\xi_{(I)}=\sum_{1\leq l_1,\cdots,l_{q-1}\leq r} a_{l_1i_1}\cdots a_{l_{q-1}i_{q-1}}\eta_{(L)}\in\mathbb{C}^r.\label{U'}
    \end{equation}
	From (\ref{diag})$\sim$(\ref{U}), it follows that 
	\begin{align*}
		Q_{q,\varphi,x}(\eta,\eta)=&\sum_{|I|=q-1}Q_{\varphi,x}(\eta_{(I)},\eta_{(I)})\\=&\sum_{\begin{subarray}{c}|I|=q-1\\1\leq l_1,\cdots,l_{q-1}\leq r\\ 1 \leq m_1,\cdots,m_{q-1}\leq r \end{subarray}}Q_{\varphi,x}(A\xi_{(L)},A\xi_{(M)})a_{i_1l_1}\cdots a_{i_{q-1}l_{q-1}}\overline{a_{i_1m_1}\cdots a_{i_{q-1}m_{q-1}}}\\ 
		=&  
        \sum_{\begin{subarray}{c}1\leq l_1,\cdots,l_{q-1}\leq r\\ 1 \leq m_1,\cdots,m_{q-1}\leq r \end{subarray}}\frac{Q_{\varphi,x}(A\xi_{(L)},A\xi_{(M)})}{(q-1)!}\sum_{1\leq i_1\cdots,i_{q-1}\leq r}a_{i_1l_1}\cdots a_{i_{q-1}l_{q-1}}\overline{a_{i_1m_1}\cdots a_{i_{q-1}m_{q-1}}} \\ =&\sum_{|L|=q-1}Q_{\varphi,x}(A\xi_{(L)},A\xi_{(L)})\\
        =&\sum_{|L|=q-1\atop 1\leq l\leq r}\lambda_l|(A\xi_{(L)})_l|^2.
    \end{align*}
    To go back from $\xi$ to $\eta$, we substitute (\ref{U'}) into the above identity
    \begin{align*}
    	Q_{q,\varphi,x}(\eta,\eta)=& \sum_{|L|=q-1\atop 1\leq l\leq r}\lambda_l\bigg(\sum_{1\leq j_1,\cdots,j_{q-1}\leq r\atop1\leq k_1,\cdots,k_{q-1}\leq r}a_{j_1l_1}\cdots a_{j_{q-1}l_{q-1}}\overline{a_{k_1l_1}\cdots a_{k_{q-1}l_{q-1}}}\eta_{lJ}\overline{\eta_{lK}}\bigg)\\ =& \sum_{\begin{subarray}{c}1\leq l\leq r\\1\leq j_1,\cdots,j_{q-1}\leq r\\ 1 \leq k_1,\cdots,k_{q-1}\leq r \end{subarray}}\frac{\lambda_l\eta_{lJ}\overline{\eta_{lK}}}{(q-1)!}\bigg(\sum_{1\leq l_1,\cdots,l_{q-1}\leq r}a_{j_1l_1}\cdots a_{j_{q-1}l_{q-1}}\overline{a_{k_1l_1}\cdots a_{k_{q-1}l_{q-1}}}\bigg) \\ =& \sum_{|K|=q-1\atop 1\leq l\leq r}\lambda_l|\eta_{lK}|^2\\
    	=& \sum_{|J|=q}\lambda_J|\eta_{J}|^2,
    \end{align*}
    where $\lambda_J=\sum_{j\in J}\lambda_j$.
\end{proof}

We will prove the following global exactness result for the complex (\ref{resolution}).
\begin{thm}\label{exactness}
	 Suppose that $P=(P_1,\cdots,P_r)$ satisfies assumptions ${\rm (A1)}\sim\text{\rm (A3)}$. For $1\leq q\leq r$, if $\Omega$ has an exhaustion function $\varphi\in C^2\left(\Omega\right)$ such that the quadratic form $Q_{\varphi,x}$ is positive semi-definite and ${\rm rank}_{\mathbb{C}}Q_{\varphi,x}\geq r-q+1$ for every $x\in \Omega$,
	then for any $ f\in L_{loc}^2\left(\Omega\right)^{\oplus\tbinom{r}{q}}$ satisfying $ \mathcal{P}_{q+1}f=0 $, there is some $ u\in L_{loc}^2\left(\Omega\right)^{\oplus\tbinom{r}{q-1}}
	 $ such that $ \mathcal{P}_{q}u=f $.
\end{thm}

Let's first recall a variant of the Riesz representation theorem(\cite{H3}) which is very useful for solving overdetermined systems.
\begin{lemma}\label{functional analysis}
	Let $ T:H_1\longrightarrow H_2 $ and $ S:H_2\longrightarrow H_3 $ be closed, densely defined linear operators such that $ {\rm Im}T\subseteq {\rm Ker}S $. Let $C>0$ be a constant, if for any $g\in{\rm Dom}\left(T^*\right)$ we have
	\begin{align*}
		|\left(g,f\right)_{H_2}|\leq C\|T^*g\|_{H_1},
	\end{align*}
	then there exists $u\in H_1$ such that $Tu=f$ and $\|u\|_{H_1}\leq C\|f\|_{H_2}$. In particular, assumme that
	\begin{equation*}
		\|g\|_{H_2}^2\leq C^2\left(\|T^*g\|_{H_1}^2+\|Sg\|_{H_3}^2\right),\quad  g\in {\rm Dom}\left(T^*\right)\cap{\rm Dom}\left(S\right).
	\end{equation*}
	If $Sf=0$, then there exists $u\in H_1$ such that $Tu=f$ and $\|u\|_{H_1}\leq C\|f\|_{H_2}$.
\end{lemma}

We will apply the above lemma to Hilbert spaces
\begin{equation}
	H_1=L_{\phi_1}^2\left(\Omega\right)^{\oplus\tbinom{r}{q-1}},  \ H_2=L_{\phi_2}^2\left(\Omega\right)^{\oplus\tbinom{r}{q}}, \ H_3 =L_{\phi_3}^2\left(\Omega\right)^{\oplus\tbinom{r}{q+1}}\label{Hil}
\end{equation}
where $1\leq q\leq r$ and $H_3:=0$ if $q=r$, and
\begin{equation}
	S=\text{the maximal extension of}\ \mathcal{P}_{q+1}, \ \ T=\text{the maximal extension of}\ \mathcal{P}_{q},\label{Ope}
\end{equation}
where $\phi_1,\phi_2,\phi_3\in C\left(\Omega\right)$ are real valued functions. The weighted $L^2$-space $L_\phi^2\left(\Omega\right)$ for a real valued function $\phi\in C\left(\Omega\right)$ is defined as  $$ L_\phi^2\left(\Omega\right)=\bigg\{ f\in L^2_{loc}\left(\Omega\right) \ \big| \ \int_\Omega |f(x)|e^{-\phi}<+\infty\bigg\}.$$We denote by $(\cdot,\cdot)_\phi$ the inner product of $L_\phi^2\left(\Omega\right)$.

Repeating the same argument in \cite{H1}(Lemma 4 on page 430), we have 
\begin{lemma}\label{dense lemma}
	Fix a sequence of real valued functions $\lbrace\eta_\nu\rbrace_{\nu=1}^{\infty} $  in $\mathcal{D}\left(\Omega\right) $ such that $ 0\leq\eta_\nu\leq1 $ for all $ \nu $, and $ \eta_\nu=1 $ on any compact subset of $ \Omega $ when $ \nu $ is large. Suppose that there is a real valued function $\tau\in C^\infty\left(\Omega\right)$ such that 
	\begin{equation}\label{tau}
		|\sigma(\mathcal{P}_{q})(\cdot,{\rm d}\eta_\nu)|^2\leq e^{\tau}\ \text{and}\ \ |\sigma(\mathcal{P}_{q+1})(\cdot,{\rm d}\eta_\nu)|^2\leq e^{\tau},\ \nu=1,2,\cdots
	\end{equation}
    hold on $\Omega$ where $ \sigma(\mathcal{P}_{q}) $ denotes the principle symbol of $ \mathcal{P}_{q} $. Then $ \mathcal{D}(\Omega)^{\oplus\tbinom{r}{q}}$ is dense in $ {\rm Dom}\left(T^*\right)\cap {\rm Dom}\left(S\right) $ for the graph norm
    \begin{equation*}
    	f\rightarrow\|f\|_{H_2}+\|T^*f\|_{H_1}+\|Sf\|_{H_3},
    \end{equation*} where $H_1,H_2,H_3$ and $S,T$ are defined by $ (\ref{Hil}) $ and $ (\ref{Ope}) $ with weight functions $\phi_1,\phi_2,\phi_3 \in C\left(\Omega\right)$ satisfying $\phi_3-\phi_2=\phi_2-\phi_1=\tau$.
\end{lemma}

\begin{proof}
	For any $f\in{\rm Dom}\left(S\right)$,
	\begin{align*}
		S(\eta_\nu f)-\eta_\nu Sf=\sigma(\mathcal{P}_{q+1})(\cdot,{\rm d}\eta_\nu)f,
	\end{align*}
	by (\ref{tau}) we have
	\begin{align*}
		|S(\eta_\nu f)-\eta_\nu Sf|^2e^{-\phi_3}\leq|f|^2e^{-\phi_2},
	\end{align*}
	According to the dominated convergence theorem, 
	\begin{align*}
		\|S(\eta_\nu f)-\eta_\nu Sf\|_{\phi_3}\longrightarrow0\ \text{when}\ \nu\rightarrow+\infty.
	\end{align*}
	
	Since $\eta_\nu f\in{\rm Dom}\left(T^*\right)$ for each $\eta_\nu$ if $f\in{\rm Dom}\left(T^*\right)$, it follows that for any $u\in{\rm Dom}\left(T\right)$
	\begin{align*}
		|\left(T^*(\eta_\nu f)-\eta_\nu T^*f,u\right)_{\phi_1}|&=|\left(f,\eta_\nu Tu-T(\eta_\nu u)\right)_{\phi_2}|\\
		&=|\left(f,\sigma(\mathcal{P}_{q})(\cdot,{\rm d}\eta_\nu)u\right)_{\phi_2}|\\
		&\leq\int_{\Omega}|f|e^{-\frac{\phi_2}{2}}|u|e^{-\frac{\phi_1}{2}},
	\end{align*}
	where the last line holds again by (\ref{tau}). This implies
	\begin{align*}
		|T^*(\eta_\nu f)-\eta_\nu T^*f|^2e^{-\phi_1}\leq|f|^2e^{-\phi_2},
	\end{align*}
	which we conclude by dominated convergence that for any $f\in{\rm Dom}\left(T^*\right)$,
	\begin{align*}
		\|T^*(\eta_\nu f)-\eta_\nu T^*f\|_{\phi_1}\longrightarrow0\ \text{when}\ \nu\rightarrow+\infty.
	\end{align*}
	Thus, $\eta_\nu f\longrightarrow f$ in the graph norm if $f\in {\rm Dom}\left(T^*\right)\cap {\rm Dom}\left(S\right)$. The proof is thus completed by Friedrichs' lemma in \cite{H4}.
\end{proof}

\begin{remark}
	Note that one can always find smooth function $\tau$ satisfying $(\ref{tau})$, since this condition means only a finite number of lower bounds for $e^\tau$ on any compact subset of $\Omega$. 
\end{remark}

For later use in the proof of Theorem \ref{exactness}, we establish the Bochner type formula as follows.

\begin{prop}\label{bochner}
	For any $f\in{\rm Dom}\left(T^*\right)\cap{\rm Dom}\left(S\right)$ and $\epsilon>0$ we have
	\begin{align*}
		&{\rm Re}\sum_{|I|=q-1}\sum_{j,k\notin 
			I}\left(\left(p_j\bar{p}_k\phi_3+e_{jk}^l\bar{p}_l\phi_3\right)f_{kI},f_{jI}\right)_{\phi_3}\\
		\leq 
		&\left(1+\epsilon\right)\left(\|T^*f\|_{\phi_1}^2+\|Sf\|_{\phi_3}^2\right)+\left(2+\frac{2}{\epsilon}\right)\|\sigma\left({}^t\mathcal{P}_q\right)\left(\cdot\,,{\rm d}\tau\right)f\|_{\phi_3}^2+\frac{1}{\epsilon}\|\kappa f\|_{\phi_3}^2,
	\end{align*}
	where  $\tau$ is a function satisfying $(\ref{tau})$, $\kappa$ is a ploynomial in $|c_{ij}^k|,|d_{ij}^k|,|e_{ij}^k|,|\nabla e_{ij}^k|$, $|\nabla a_j^\nu|$ and $|\nabla^2a_j^\nu|$ with coefficients in $\mathbb{R}_{\geq0}(1\leq i,j,k\leq r,0\leq\nu\leq n)$. 
\end{prop}

\begin{proof}
	Let $ ^tP_{j,\phi_3} $ be the formal adjoint of $ P_j $ with respect to the inner product $ (\cdot,\cdot)_{\phi_3} $, i.e., 
	$$ ^tP_{j,\phi_3}=e^{\phi_3}\circ {}^tP_j\circ e^{-\phi_3}. $$ 
	From the expressions in (\ref{boundary operator}) and (\ref{formal adjoint of boundary operator}) for $f\in\mathcal{D}\left(\Omega\right)^{\oplus\tbinom{r}{q}}$ we read off 
	\begin{align}
		T^*f&=e^{\phi_1}\circ{}^{t}\mathcal{P}_q\circ e^{-\phi_2}f\nonumber\\
		&=e^{-\tau}\bigg(\sum_{j\notin I}{}^tP_{j,\phi_3}f_{jI}\bigg)_{|I|=q-1}+e^{-\tau}\sigma\left({}^t\mathcal{P}_q\right)\left(\cdot\,,{\rm d}\tau\right)f+e^{-\tau}\kappa_1 f,\label{T*m}\\ 
		Sf&=\bigg(\sum_{k\in K}(-1)^{(k,K\setminus k)}P_kf_{K\setminus k}\bigg)_{|K|=q+1}+\kappa_1 f,\label{Sm}
	\end{align}
	where we denote by $\kappa_1$ various polynomials in $c_{ij}^k$ with constant coefficients for $1\leq i,j,k\leq r$. Set
	\begin{align*}
		A_1&=\sum_{|I|=q-1}\bigg\|\sum_{j\notin I}{}^tP_{j,\phi_3}f_{jI}\bigg\|_{\phi_3}^2,\\
		A_2&=\sum_{|K|=q+1}\bigg\|\sum_{k\in K}(-1)^{(k,K\setminus k)}P_kf_{K\setminus k}\bigg\|_{\phi_3}^2.
	\end{align*} 
	We will proceed by estimating $A_1+A_2$ from both sides. In view of (\ref{T*m}) and (\ref{Sm}), it is easy to see that for any $\epsilon>0 $
	\begin{align}\label{upper}
		A_1+A_2\leq& \left(1+\epsilon\right)\left(\|T^*f\|_{\phi_1}^2+\|Sf\|_{\phi_3}^2\right)\nonumber\\
		&+\left(2+\frac{2}{\epsilon}\right)\|\sigma\left({}^t\mathcal{P}_q\right)\left(\cdot\,,{\rm d}\tau\right)f\|_{\phi_3}^2+\frac{1}{\epsilon}\|\kappa f\|_{\phi_3}^2.
	\end{align} 
	Here and hereafter, we denote by $\kappa$ various polynomials in $|c_{ij}^k|$, $|d_{ij}^k|$, $|e_{ij}^k|$,  $|\nabla e_{ij}^k|$, $|\nabla a_j^\nu|$ and $|\nabla^2a_j^\nu|$ with coefficients in $\mathbb{R}_{\geq0}$ where $1\leq i,j,k\leq r,\ 0\leq\nu\leq n$. We rewrite  $ A_1+A_2 $ by integration by parts as follows
	\begin{align}
		A_1+A_2
		=&\sum_{|I|=q-1}\sum_{j,k\notin 
			I}\left({}^tP_{k,\phi_3}f_{kI},{}^tP_{j,\phi_3}f_{jI}\right)_{\phi_3}+\sum_{|K|=q+1}\sum_{k\in
			K}\|P_kf_{K\setminus k}\|_{\phi_3}^2\notag\\
		&+\sum_{|K|=q+1}\sum_{j,k\in K,j\neq k}(-1)^{(k,K\setminus k)+(j,K\setminus j)}\left({}^tP_{k,\phi_3}P_jf_{K\setminus j},f_{K\setminus k}\right)_{\phi_3}\notag\\
		=&\sum_{|I|=q-1}\sum_{j,k\notin 
			I}\left(P_j{}^tP_{k,\phi_3}f_{kI},f_{jI}\right)_{\phi_3}+\sum_{|J|=q}\sum_{k\notin
			J}\|P_kf_J\|_{\phi_3}^2\notag\\
		&-\sum_{|I|=q-1}\sum_{j,k\notin I,j\neq k}\left({}^tP_{k,\phi_3}P_jf_{kI},f_{jI}\right)_{\phi_3}\notag\\
		=&\sum_{|I|=q-1}\sum_{j,k\notin 
			I}\left(P_j{}^tP_{k,\phi_3}f_{kI},f_{jI}\right)_{\phi_3}+\sum_{|J|=q}\sum_{k\notin
			J}\|P_kf_J\|_{\phi_3}^2\notag\\  
		&-\sum_{|I|=q-1}\sum_{j,k\notin
			I}\left({}^tP_{k,\phi_3}P_jf_{kI},f_{jI}\right)_{\phi_3}+\sum_{|I|=q-1}\sum_{j\notin
			I}\left({}^tP_{j,\phi_3}P_jf_{jI},f_{jI}\right)_{\phi_3}\notag\\
		=&\sum_{|I|=q-1}\sum_{j,k\notin 
			I}\left(\left[P_j,{}^tP_{k,\phi_3}\right]f_{kI},f_{jI}\right)_{\phi_3}+\sum_{|J|=q}\sum_{j=1}^{r}\|P_jf_J\|_{\phi_3}^2.\label{a1+a2}
	\end{align}
	In the second equality, we changed the indices of summation by letting $ 
	J=K\setminus k,\,I=K\setminus{\lbrace j,k\rbrace} $, and making use of the identity $ (-1)^{(j,Ik)+(k,Ij)}=-(-1)^{(j,I)+(k,I)} $ for $ j,k\notin I,\,j\neq 
	k $. The equality (\ref{a1+a2}) allows us to estimate $A_1+A_2$ from below
	\begin{align}\label{inequality}
		{\rm Re}\sum_{|I|=q-1}\sum_{j,k\notin I}\left(-\left[p_j,F_{k}\right]f_{kI},f_{jI}\right)_{\phi_3}+B_1\leq A_1+A_2+\|\kappa f\|_{\phi_3}^2,
	\end{align}
	where $F_k:=e^{\phi_3}\circ\bar{p}_k\circ e^{-\phi_3}$ and $ B_1:=\sum_{|J|=q}\sum_{j=1}^{r}\|p_jf_J\|_{\phi_3}^2 $. 
	
	By virtue of $ F_k=\bar{p}_k-\sigma(\bar{p}_k)(\cdot\,,{\rm d}\phi_3) $, we have
	\begin{equation*}
		-\left[p_j,F_k\right]=-\left[p_j,\bar{p}_k-\sigma(\bar{p}_k)(\cdot\,,{\rm d}\phi_3)\right]=-\left[p_j,\bar{p}_k\right]+\left[p_j,\sigma(\bar{p}_k)(\cdot\,,{\rm d}\phi_3)\right],
	\end{equation*}
	which in turn implies
	\begin{align}
		&{\rm Re}\sum_{|I|=q-1}\sum_{j,k\notin 
			I}\left(-\left[p_j,F_k\right]f_{kI},f_{jI}\right)_{\phi_3}\notag\\
		=&{\rm Re}\sum_{|I|=q-1}\sum_{j,k\notin 
			I}\left(\left(-\left[p_j,\bar{p}_k\right]f_{kI},f_{jI}\right)_{\phi_3}+\left(\left[p_j,\sigma(\bar{p}_k)(\cdot\,,{\rm d}\phi_3)\right]f_{kI},f_{jI}\right)_{\phi_3}\right)\notag\\
		=&{\rm Re}\sum_{|I|=q-1}\sum_{j,k\notin 
			I}\left(\left(-\left[p_j,\bar{p}_k\right]f_{kI},f_{jI}\right)_{\phi_3}+\left(p_j{}\bar{p}_k\phi_3 
		f_{kI},f_{jI}\right)_{\phi_3}\right).\label{Re}
	\end{align} 
	The last equality is valid by the following observation
	\begin{align*}
		\left[p_j,\sigma(\bar{p}_k)(\cdot,\,{\rm d}\phi_3)\right]f&=\sum_{\mu=1}^{n}\sum_{\nu=1}^{n}\left[a_j^\mu(x)\partial_\mu,\bar{a}_k^\nu(x)\partial_\nu\phi_3\right]f\\
		&=\sum_{\mu=1}^{n}\sum_{\nu=1}^{n}a_j^\mu(x)\partial_\mu(\bar{a}_k^\nu(x)\partial_\nu\phi_3 f)-\bar{a}_k^\nu(x)\partial_\nu\phi_3 a_j^\mu(x)\partial_\mu f\\
		&=\sum_{\mu=1}^{n}\sum_{\nu=1}^{n}a_j^\mu(x)\partial_\mu(\bar{a}_k^\nu(x))\partial_\nu\phi_3 f+a_j^\mu(x)\bar{a}_k^\nu(x)\partial_\mu\partial_\nu(\phi_3)f\\
		&=\sum_{\mu=1}^{n}\sum_{\nu=1}^{n}a_j^\mu(x)\partial_\mu(\bar{a}_k^\nu(x)\partial_\nu\phi_3)f\\
		&=p_j\bar{p}_k\phi_3 f.
	\end{align*}
	Substituting ($ \ref{condition of the Lie braket} $) and (\ref{Re}) into (\ref{inequality}) gives 
	\begin{align*}
		{\rm Re}\sum_{|I|=q-1}\sum_{j,k\notin  I}\left(\left(p_j\bar{p}_k\phi_3+e_{jk}^l\bar{p}_l-d_{jk}^lp_l\right)f_{kI},f_{jI}\right)_{\phi_3}+B_1\leq 
		A_1+A_2+\|\kappa f\|_{\phi_3}^2,
	\end{align*}which implies by another integration by parts involving $ \bar{p}_lf_{kI} $ 
	\begin{align}
		{\rm Re}\sum_{|I|=q-1}\sum_{j,k\notin I}\left(\left(p_j\bar{p}_k\phi_3+e_{jk}^l\bar{p}_l\phi_3\right)f_{kI},f_{jI}\right)_{\phi_3}+B_1\notag\\
		\leq
		A_1+A_2+\|\kappa f\|_{\phi_3}^2+B_1^{\frac{1}{2}}\|\kappa f\|_{\phi_3}.\label{lower}
	\end{align}
	Combining (\ref{upper}) and (\ref{lower}), it follows that
	\begin{align*}
		&{\rm Re}\sum_{|I|=q-1}\sum_{j,k\notin I}\left(\left(p_j\bar{p}_k\phi_3+e_{jk}^l\bar{p}_l\phi_3\right)f_{kI},f_{jI}\right)_{\phi_3}\nonumber\\
		\leq 
		&\left(1+\epsilon\right)\left(\|T^*f\|_{\phi_1}^2+\|Sf\|_{\phi_3}^2\right)+\left(2+\frac{2}{\epsilon}\right)\|\sigma\left({}^t\mathcal{P}_q\right)\left(\cdot\,,{\rm d}\tau\right)f\|_{\phi_3}^2+\frac{1}{\epsilon}\|\kappa f\|_{\phi_3}^2.
	\end{align*}
	According to Lemma \ref{dense lemma}, the desired estimate holds for any $f\in{\rm Dom}\left(T^*\right)\cap{\rm Dom}\left(S\right)$.
\end{proof}

We are now in a position to prove Theorem \ref{exactness}.

\begin{proof}[\bf Proof of Theorem \ref{exactness}]
	Let $\tau\in C^\infty\left(\Omega\right)$ be a function satisfying (\ref{tau}). We will make use of weight functions $ \phi_1=\chi(\varphi)-2\tau,\phi_2=\chi(\varphi)-\tau,\phi_3=\chi(\varphi) $, where $ \chi\in C^\infty(\mathbb{R}) $ is a convex increasing function to be chosen in the sequel. 
	
	The form (\ref{quadratic form}) for $\chi(\varphi)$
	becomes
	\begin{align}\label{cqf}
		\chi'(\varphi)Q_{\varphi,x}(\xi,\xi)+
		\chi''(\varphi)|\xi_jp_j\varphi|^2,
	\end{align}
	where the second term is positive since $\chi$ is convex. According to Proposition \ref{bochner}, for any $f\in{\rm Dom}\left(T^*\right)\cap{\rm Dom}\left(S\right)$ we have
	\begin{align}
		&{\rm Re}\sum_{|I|=q-1}\sum_{j,k\notin I}\left(\chi'\left(\varphi\right)\left(p_j\bar{p}_k\varphi+e_{jk}^l\bar{p}_l\varphi\right)f_{kI},f_{jI}\right)_{\phi_3}\nonumber\\
		\leq 
		&\left(1+\epsilon\right)\left(\|T^*f\|_{\phi_1}^2+\|Sf\|_{\phi_3}^2\right)+\left(2+\frac{2}{\epsilon}\right)\|\sigma\left({}^t\mathcal{P}_q\right)\left(\cdot\,,{\rm d}\tau\right)f\|_{\phi_3}^2+\frac{1}{\epsilon}\|\kappa f\|_{\phi_3}^2.\label{same}
	\end{align}
    
    Since ${\rm rank}_{\mathbb{C}}Q_{\varphi,x}\geq r-q+1$, Lemma \ref{relationship of eigenvalues} shows that  there exists some positive $\sigma\in C\left(\Omega\right)$ such that (\ref{same}) is bounded from below by $$\int_\Omega \sigma\chi'(\varphi) |f|^2e^{-\chi(\varphi)}.$$ As  $ \varphi $ is an exhaustion function on $ \Omega $, we can choose $ \chi $ increasing rapidly such that 
    $$\sigma\chi'(\varphi)\geq \bigg(\frac{\kappa^2}{\epsilon}+\left(2+\frac{2}{\epsilon}\right)|\sigma\left({}^t\mathcal{P}_q\right)\left(\cdot\,,{\rm d}\tau\right)|^2+(1+\epsilon)\bigg)e^\tau,$$
    which gives the desired estimate in Lemma \ref{functional analysis}.
\end{proof}

\begin{definition}
	$\Omega$ is said to be $P$-convex, if there exists an exhaustion function $\varphi\in C^2\left(\Omega\right)$ such that $Q_{\varphi,x}$ is positive define for every $x\in \Omega$.
\end{definition}

Given a point $y\in \Omega$, let $\varphi\in C^2\left(\Omega\right)$ be an arbitrary function vanishing at $y$ up to first order derivatives, then we have
\begin{align}
	Q_{\varphi,y}(\xi,\xi)&={\rm Re}\left(p_j\bar{p}_k\varphi(y)\right)\xi_{j}\bar{\xi}_{k}\nonumber\\
	&={\rm Re}\left(a_j^\mu(y)\left(\partial_\mu\overline{a_k^\nu(y)}\right)\partial_\nu \varphi(y)+a_j^\mu(y)\overline{a_k^\nu(y)}\partial_\mu\partial_\nu \varphi(y)\right)\xi_{j}\bar{\xi}_{k}\nonumber\\
	&=a_j^\mu(y)\xi_j\overline{a_k^\nu(y)\xi_k}\partial_\mu\partial_\nu \varphi(y)\label{local}.
\end{align}
The above computation suggests the following assumption.

$\mathbf{(A2)^*}$ ${\rm rank}_{\mathbb{C}}\left(a_j^\nu(x)\right)_{1\leq \nu\leq n\atop1\leq j\leq r}=r$ holds for all $x\in\Omega$.\\
In view of (\ref{J'}), (A2)* is stronger than previous condition (A2).

Under the assumption (A2)*, it is clear from (\ref{local}) that for each point $y\in\Omega$, there is an open ball $B(y,\varepsilon)$ centered at $y$ of radius $\varepsilon$ such that the form $Q_{|x-y|^2,x}(\xi,\xi)$ is positive at $x\in B(y,\varepsilon)$, then the open ball $B(y,\varepsilon)$ is a $P$-convex neighborhood around $y$. As a consequence of (\ref{cqf}),
$$\varphi(x):=-\log(\varepsilon^2-|x-y|^2)$$
is an exhaustion function with positive definite quadratic form on $B(y,\varepsilon)$, which yields that $B(y,\varepsilon)$ is $P$-convex. Therefore we have the following application of Theorem \ref{exactness}.

\begin{cor}\label{local existence}
	Assume that $P=(P_1,\cdots,P_r)$ satisfies {\rm (A1)}, {\rm (A2)*} and {\rm (A3)}, we have
	\begin{enumerate}
		\item[(i)] The complex $ (\ref{resolution}) $ is a resolution of the solution sheaf $\mathcal{S}_P$ and the isomorphisms $ (\ref{ISO}) $ hold.
	    \item[(ii)] If there is an exhaustion function $\varphi\in C^2\left(\Omega\right)$ such that $Q_{\varphi,x}$ is positive semi-definite, ${\rm rank}_{\mathbb{C}}Q_{\varphi,x}\geq r-q+1$ for every $x\in \Omega$ and some $1\leq q\leq r$,  then $H^q(\Omega,\mathcal{S}_P)=0$. In particular, if $\Omega$ is $P$-convex then $H^q(\Omega,\mathcal{S}_P)=0$ for all $q\geq 1$.
	\end{enumerate}
\end{cor}

In the special case where $P_j(1\leq j \leq n)$ are complex vector fields(i.e., $P_j=p_j$), we have the following application to Treves complex.

\begin{cor}\label{lst}
	For every formally integrable structure $E\subseteq \mathbb{C}TX$ with vanishing Levi form, we have
	\begin{enumerate}
		\item[(i)] The Treves complex \eqref{new complex} always has local $L^2$-solvability, i.e.,
		for any point in $X$, there is an open neighborhood $U$ satisfying 
		\begin{align*}
			\forall f\in L_{loc}^2(U,\mathfrak{U}_E^q)\ \text{with}\ {\rm d}'f=0,\ \exists u\in L_{loc}^2(U,\mathfrak{U}_E^{q-1})\ \text{such that}\ {\rm d}'u=f\ \text{in}\ U
		\end{align*}
		for $1\leq q\leq r$, where $L_{loc}^2(U,\mathfrak{U}_E^q)$ is defined by
		\begin{align*}
			L_{loc}^2(U,\mathfrak{U}_E^q)=\{g\ {\rm mod}\ \mathfrak{N}_E^q(U)\ |\ g\in L_{loc}^2(U,\Lambda^q\mathbb{C}T^*X)\}.
		\end{align*}
		\item[(ii)] If we assume further that $E+\bar{E}$ has constant rank, then the Treves complex \eqref{new complex} is smoothly solvable in the sense of germs, i.e., for any point in $X$, there exist open neighborhoods $U'\subseteq U$ of that point satisfying 
		\begin{align*}
			\forall f\in\mathfrak{U}_E^q(U)\ \text{with}\ {\rm d}'f=0,\ \exists u\in\mathfrak{U}_E^{q-1}(U')\ \text{such that}\ {\rm d}'u=f\ \text{in}\ U'
		\end{align*}
		for $1\leq q\leq r$.
	\end{enumerate}
\end{cor}
\begin{proof}
	(i) Since $E\subseteq \mathbb{C}TX$ is a formally integrable subbundle, assumptions (A1) and (A2)* are automatically fulfilled. The assumption (A3) is equivalent to $\left[\Gamma(U,E),\Gamma(U,\bar{E})\right]\subseteq\Gamma(U,E)+\Gamma(U,\bar{E})$ which, by a simple linear algebra observation, means that $E$ has vanishing Levi form. Thus, the local $L^2$-solvability of the Treves complex follows directly from (\ref{T1}), (\ref{T2}) and the same arguments used in the proof of Corollary \ref{local existence}.
	
	(ii) If $E+\bar{E}$ has constant rank, i.e., it is a subbundle of $\mathbb{C}TX$, then $\left[\Gamma(U,E),\Gamma(U,\bar{E})\right]\subseteq\Gamma(U,E)+\Gamma(U,\bar{E})$ implies that such a formally integrable subbundle $E$ is a Levi flat structure, and therefore is locally integrable according to Nirenberg's theorem(\cite{N}). By means of Theorem VIII.9.1 in \cite{T2}, we obtain the smooth solvability in the sense of germs.
\end{proof}

There are obstructions to the local solvability of the Treves complex(cf. \cite{T2}, \cite{HM}, \cite{BCH}, \cite{Nf}, \cite{FF} and references therein), even if the subbundle $E$ is locally integrable. When the Levi form is non-degenerate, the so-called $Y(q)$-condition was shown to be the necessary and sufficient for  local solvability of the Treves complex(\cite{T1.5}). It was proven in \cite{FF} that the $Y(q)$-condition is also necessary for local solvability in the sense of ultra-distributions. There was a long-standing conjecture regarding the characterization of the local solvability of the Treves complex in terms of the reduced homology of the regular level sets of a first integral for locally integrable structures of corank one, with the characterization not involving any restrictions on the Levi forms of the underlying structures. The complete proof of this conjecture was presented in the influential paper \cite{CH}.

\section{Cousin problems}\label{s4}
Throughout this section, we assume that the conditions (A1), (A2)* and (A3) for differential operators $P_1,\cdots,P_r$ are fulfilled. We will consider both additive and multiplicative Cousin problems for the system 
\begin{equation}
	P_1 u=\cdots =P_ru=0.\label{hom}
\end{equation} 
We begin with the formulation of Cousin problems for the system.

\begin{definition}\label{Cousin I}
An additive Cousin datum of the system $(\ref{hom}) $ is a collection $\{\Omega_\alpha,u_{\alpha\beta}\}_{\alpha,\beta\in\Lambda}$ where $\{\Omega_\alpha\}_{\alpha\in\Lambda}$ is an open covering of $\Omega$, each $u_{\alpha\beta}\in L^2_{loc}(\Omega_\alpha\cap\Omega_\beta)$ satisfies $ (\ref{hom}) $ on $\Omega_\alpha\cap\Omega_\beta(\alpha,\beta\in\Lambda)$ and $$u_{\alpha\beta}+u_{\beta\gamma}+u_{\gamma\alpha}=0  \ {\rm on} \  \Omega_\alpha\cap\Omega_\beta\cap\Omega_\gamma(\alpha,\beta,\gamma\in\Lambda).$$
The Cousin problem for an additive Cousin datum $\{\Omega_\alpha,u_{\alpha\beta}\}_{\alpha,\beta\in\Lambda}$ is to find a sequence of solutions $u_\alpha\in L_{loc}^2(\Omega_\alpha)$ of $ (\ref{hom}) $, such that $u_{\alpha\beta}=u_\beta-u_\alpha$ hold on $\Omega_\alpha\cap\Omega_\beta$ for all $\alpha,\beta\in\Lambda$.
\end{definition}

We can construct global solutions of the overdetermined system $(\ref{OS}) $ from local ones by using solutions of the additive Cousin problem. 
\begin{prop}
Assume that the Cousin problem is always solvable for every additive Cousin datum of the system $(\ref{hom})$. Let $\{\Omega_\alpha\}_{\alpha\in\Lambda}$ be an open covering of $\Omega$ and  $v_\alpha\in \mathcal{D}'(\Omega_\alpha)$ be a solution of the system  $ (\ref{OS}) $$(\alpha\in\Lambda)$. If the difference $v_\alpha-v_\beta\in L^2_{loc}(\Omega_\alpha\cap\Omega_\beta)$ for all $\alpha,\beta\in\Lambda$, then there is a global solution $v\in L^2_{loc}\left(\Omega\right)$ of the system $ (\ref{OS}) $ such that $v|_{\Omega_\alpha}-v_\alpha\in L^2_{loc}(\Omega_\alpha)$ for each $\alpha\in\Lambda$.
\end{prop}

\begin{proof}
According to Definition \ref{Cousin I},  $\{\Omega_\alpha,(v_{\alpha}-v_{\beta})|_{\Omega_\alpha\cap\Omega_\beta}\}_{\alpha,\beta\in\Lambda}$  is an additive Cousin datum of  the system $(\ref{hom})$. Let $\{u_\alpha\}_{\alpha\in\Lambda}$ be a solution of the above additive Cousin datum, then $$u_\beta-u_\alpha = v_\alpha-v_\beta$$ hold on $\Omega_\alpha\cap\Omega_\beta$ for all $\alpha,\beta\in\Lambda.$ Hence, there is a global solution $v\in\mathcal{D}'\left(\Omega\right)$ of the system $(\ref{OS})$ such that
\begin{equation*}
	v|_{\Omega_\alpha} = u_\alpha+v_\alpha
\end{equation*} 
for each $\alpha\in\Lambda$.
\end{proof}

The assumption (A1) is equivalent to  $$[p_j,p_k]=c_{jk}^lp_l, \ \  p_ja^0_k-p_ka^0_j=c^l_{jk}a^0_l\  {\rm for} \ 1\leq i,j\leq r,$$ and therefore one can find(by Corollary \ref{local existence}) a function $\eta\in L_{loc}^2\left(\Omega\right)$ such that for $ 1\leq j\leq r$ 
\begin{equation}
	p_j\eta=a^0_j\label{eta1}
\end{equation}
when $\Omega$ is assumed to be $P$-convex. We will fix such an $\eta$ satisfying (\ref{eta1}) in the sequel and formulate the multiplicative Cousin Problem as below.

\begin{definition}
A multiplicative Cousin datum of $(\ref{hom}) $ is a collection $\{\Omega_\alpha,u_{\alpha\beta}\}_{\alpha,\beta\in\Lambda}$ where $\{\Omega_\alpha\}_{\alpha\in\Lambda}$ is an open covering of $\Omega$, each $u_{\alpha\beta}\in C^0(\Omega_\alpha\cap\Omega_\beta)$  satisfies $ (\ref{hom}) $ on $\Omega_\alpha\cap\Omega_\beta(\alpha,\beta\in\Lambda)$ and $$e^{3\eta}u_{\alpha\beta}u_{\beta\gamma}u_{\gamma\alpha}=1  \ {\rm on} \  \Omega_\alpha\cap\Omega_\beta\cap\Omega_\gamma(\alpha,\beta,\gamma\in\Lambda).$$
The Cousin problem for a multiplicative Cousin datum $\{\Omega_\alpha,u_{\alpha\beta}\}_{\alpha,\beta\in\Lambda}$ is to find a sequence of solutions $u_\alpha\in C^0(\Omega_\alpha)$ of  $ (\ref{hom}) $, such that each $u_\alpha$ is nowhere zero on $\Omega_\alpha$ and $$e^\eta u_{\alpha\beta}=\frac{u_\beta}{u_\alpha}$$ hold on $\Omega_\alpha\cap\Omega_\beta$ for all $\alpha,\beta\in\Lambda$, where $\eta$ is a solution of $ (\ref{eta1}) $.
\end{definition}

The solvability of the multiplicative Cousin problem enables us to patch together local solutions of the system (\ref{hom}) as follows. 
\begin{prop}
Assume that the Cousin problem is always solvable for every multiplicative Cousin datum of the system $(\ref{hom})$ and that the function $\eta\in C^0\left(\Omega\right)$. Let $\{\Omega_\alpha\}_{\alpha\in\Lambda}$ be an open covering of $\Omega$ and $v_\alpha\in C^0(\Omega_\alpha)$ be a solution  of $ (\ref{hom}) $ for each $\alpha\in\Lambda$, such that the quotients  
$\frac{v_\alpha}{v_\beta}\in C^0(\Omega_\alpha\cap\Omega_\beta) \ {\rm are \ nowhere \ zero}$
for all $\alpha,\beta\in\Lambda$, then there is a global solution $v\in C^0\left(\Omega\right)$ of the system $ (\ref{hom}) $ such that $\frac{v|_{\Omega_\alpha}}{v_\alpha} \in C^0(\Omega_\alpha)$ is nowhere zero for each $\alpha\in\Lambda$.
\end{prop}

\begin{proof}
We rewrite the identity (\ref{eta1}) as 
\begin{equation}
	P_j=e^{-\eta}p_j(e^\eta \ \cdot),  \ 1\leq j\leq r.\label{eta}
\end{equation} 
Then, it is easy to see that for $ 1\leq j\leq r $,
\begin{align*}
	P_j(\frac{v_\alpha}{e^{\eta}v_\beta})& =e^{-\eta}p_j(\frac{e^\eta v_\alpha}{e^{\eta}v_\beta})\\ &= \frac{P_jv_\alpha}{e^\eta v_\beta}-\frac{v_\alpha P_jv_\beta}{e^\eta v_\beta^2}=0 
\end{align*}
hold on $\Omega_\alpha\cap\Omega_\beta$ for all $\alpha,\beta\in\Lambda$. Consequently, we know that
\begin{equation}
	\{\Omega_\alpha,\big(\frac{v_{\alpha}}{e^{\eta}v_{\beta}}\big)|_{\Omega_\alpha\cap\Omega_\beta}\}_{\alpha,\beta\in\Lambda}\label{datum}
\end{equation}  
is a multiplicative Cousin datum of  the system $(\ref{hom})$. Let $\{u_\alpha\}_{\alpha\in\Lambda}$ be a solution of the above multiplicative Cousin datum (\ref{datum}), then 
$$\frac{u_\beta}{u_\alpha} = \frac{v_\alpha}{v_\beta}$$ hold
on $\Omega_\alpha\cap\Omega_\beta$ for all $\alpha,\beta\in\Lambda$. So, one can find a global function $v\in C^\infty\left(\Omega\right)$ satisfying 
\begin{equation*}
	v|_{\Omega_\alpha} =e^\eta u_\alpha v_\alpha
\end{equation*} 
for each $\alpha\in\Lambda$. Again, we know by (\ref{eta}) that on each $\Omega_\alpha$ 
\begin{align*}
	P_jv&=P_j(e^\eta u_\alpha v_\alpha)=e^{-\eta}p_j(e^{2\eta}u_\alpha v_\alpha)\\ &=e^\eta\left(v_\alpha P_ju_\alpha + u_\alpha P_jv_\alpha\right) =0 
\end{align*}
for $1\leq j\leq r$.
\end{proof}

Global exactness in Section \ref{s3}(Corollary \ref{local existence}) enables us to provide solutions of Cousin problems by a purely sheaf-theoretic approach.
\begin{thm}\label{Cousin}
When $\Omega$ is $P$-convex, we have 
\begin{enumerate}
	\item[(i)] The additive Cousin problem is always solvable.
    \item[(ii)] Under the condition that the principal part $(p_1,\cdots,p_r)$ of $P$ is $C^0$-hypoelliptic, the multiplicative Cousin problem is always solvable if and only if $H^2(\Omega,\mathbb{Z})=0$. Here, the $C^0$-hypoellipticity of $(p_1,\cdots,p_r)$ means that for any $u\in L^2_{loc}\left(\Omega\right)$ if $p_ju\in C^\infty\left(\Omega\right)$ for $1\leq j\leq r$ then $u\in C^0\left(\Omega\right)$.
\end{enumerate}	 
\end{thm}

\begin{proof}
(i) In view of Definition \ref{Cousin I}, we know that an additive Cousin datum $\{\Omega_\alpha,u_{\alpha\beta}\}_{\alpha,\beta\in\Lambda}$ defines an element 
\begin{equation*}
	[\{u_{\alpha\beta}\}_{\alpha,\beta\in\Lambda}]\in H^1(\{\Omega_\alpha\}_{\alpha\in\Lambda},\mathcal{S}_P),
\end{equation*} 
and that the Cousin problem for the datum $\{\Omega_\alpha,u_{\alpha\beta}\}_{\alpha,\beta\in\Lambda}$ has a solution if and only if $0=[\{u_{\alpha\beta}\}_{\alpha,\beta\in\Lambda}]\in H^1(\{\Omega_\alpha\}_{\alpha\in\Lambda},\mathcal{S}_P).$ Therefore, (1) follows from Corollary \ref{local existence}.

(ii) Under the hypoellipticity assumption, the function $\eta$ satisfying (\ref{eta1}) must be continuous.
As a consequence of the identity (\ref{eta}), $C^0$-hypoellipticity of the principal part $(p_1,\cdots,p_r)$ implies
\begin{equation}
	\mathcal{S}_P\subseteq\mathcal{C},\label{elliptic}
\end{equation} 
where $\mathcal{C}$ is the sheaf of germs of continuous functions over $\Omega$. 

We define a multiplication on $\mathcal{C}$ as follows
\begin{equation}
	u\star v:=e^\eta uv.\label{star}
\end{equation}
By virtue of (\ref{eta}), we have  for $1\leq j\leq r$
\begin{equation} 
	P_j(u\star v)=P_ju\star v+u\star P_jv.\label{leb}
\end{equation}
Since the multiplicative identity of $\star$ is given by $\frac{1}{e^{\eta}}$ which is annihilated by every $P_j$(by (\ref{eta})), it follows from (\ref{leb}) that $\star$ induces a multiplication on $\mathcal{S}_P$.  In order to deal with the multiplicative Cousin problem, we introduce(by using (\ref{elliptic})) the sheaf $\mathcal{S}^\star_P$ of invertible germs(according to the multiplication $\star$) of $\mathcal{S}_P$. By definition, $\mathcal{S}^\star_P$ is a sheaf of Abelian groups with respect to the multiplication $\star$ defined by (\ref{star}). It is clear that  $\mathcal{S}^\star_P$ consists of germs of $\mathcal{S}_P$ which are nowhere zero. Moreover, for any $u\in \mathcal{S}^\star_P$, we denote by $u^{-1}$ the inverse of $u$ with respect to the multiplication $\star$. The definition of $\star$ yields
\begin{equation}
	u^{-1}=\frac{1}{e^{2\eta} u}.\label{inverse}
\end{equation}

Again by the definition (\ref{star}) of the multiplication $\star$, we have the following sheaf-morphism 
\begin{eqnarray*}
	{\rm Exp}^\eta: \mathcal{S}_P & \rightarrow & \mathcal{S}_P^\star \notag \\ u &\mapsto &e^{2\pi\sqrt{-1}e^\eta u-\eta},
\end{eqnarray*} 
whose kernel $\mathcal{K}_{{\rm Exp}^\eta}\cong$ the constant sheaf $\Lambda$ over $\Omega$. It gives us the following twisted exponential sequence 
\begin{equation}
	0\longrightarrow \Lambda\stackrel{e^{-\eta}\cdot}{\longrightarrow}\mathcal{S}_P\stackrel{{\rm Exp}^\eta}{\longrightarrow}\mathcal{S}_P^\star\longrightarrow0,\label{twisted}
\end{equation}
which is exact because of (\ref{elliptic}). The twisted exponential sequence (\ref{twisted}) implies the following exact sequence
\begin{equation}
	H^1(\Omega,\mathcal{S}_P)\rightarrow H^1(\Omega,\mathcal{S}_P^\star)\rightarrow H^2(\Omega,\mathbb{Z})\rightarrow H^2(\Omega,\mathcal{S}_P).
	\label{e3}
\end{equation} 
From Corollary \ref{local existence}, it follows that $$H^1(\Omega, \mathcal{S}_P)=H^2(\Omega, \mathcal{S}_P)=0$$ since $\Omega$ is $P$-convex. Hence, we obtain an isomorphism from (\ref{e3}) 
\begin{equation} 
	H^1(\Omega,\mathcal{S}_P^\star)\cong H^2(\Omega,\mathbb{Z}).\label{is}
\end{equation}

For every multiplicative Cousin datum $\{\Omega_\alpha,u_{\alpha\beta}\}_{\alpha,\beta\in\Lambda}$,
$$u_{\alpha\beta}\star u_{\beta\gamma}\star u_{\gamma\alpha}=\frac{1}{e^\eta}({\rm the \ multiplicative \ identity \ of} \ \star )$$ 
on $\Omega_\alpha\cap\Omega_\beta\cap\Omega_\gamma(\alpha,\beta,\gamma\in\Lambda)$, i.e., $\{\Omega_\alpha,u_{\alpha\beta}\}_{\alpha,\beta\in\Lambda}$ defines an element 
\begin{equation*}
	[\{u_{\alpha\beta}\}_{\alpha,\beta\in\Lambda}]\in H^1(\{\Omega_\alpha\}_{\alpha\in\Lambda},\mathcal{S}^\star_P).
\end{equation*} 
By the identity (\ref{inverse}), $\{u_\alpha\}_{\alpha\in\Lambda}$ is a solution of the multiplicative Cousin datum $\{\Omega_\alpha,u_{\alpha\beta}\}_{\alpha,\beta\in\Lambda}$ exactly means that $u_\alpha\in \Gamma(\Omega_\alpha,\mathcal{S}_P^\star)$ and 
\begin{equation}
	u_{\alpha\beta}=\frac{u_\beta}{e^\eta u_\alpha}=u_\beta\star(u_\alpha)^{-1}.\label{sol}
\end{equation} 
The desired conclusion follows from (\ref{is})$\sim$(\ref{sol}).
\end{proof}



\bigskip
Finally, we mention another application of the additive Cousin problem for an elliptic differential operator $P=(P_1,\cdots,P_r)$, i.e., $\{p_1,\cdots,p_r\}$ spans an elliptic structure over $\Omega$ (see page 15 in \cite{BCH} for the definition of elliptic structures). By the ellipticity assumption, we can choose some solution $\eta\in C^\infty\left(\Omega\right)$ of the system (\ref{eta1}) if $\Omega$ is assumed to be $P$-convex.

Let $F\in C^\infty\left(\Omega\right)$ be a solution of the system $P_1F=\cdots =P_rF=0$ with the property that $${\rm d}F\neq 0 \ {\rm holds \ on}\  Z:=F^{-1}(0).$$ Obviously, 
\begin{align}\label{eF}
        {\rm d}\left(e^\eta F\right)=e^\eta {\rm d}F\neq 0\ {\rm on}\ Z=\left(e^\eta F\right)^{-1}(0).
\end{align}
Under the ellipticity assumption, it is known that at each point of $\Omega$ there is a coordinate chart
\begin{align}\label{lcn.}
	(U; x_1,\cdots, x_{r-s}, y_1, \cdots, y_{r-s}, t_1,\cdots, t_s)\ (2r-s=n)
\end{align}
such that 
\begin{equation}
	\{p_1,\cdots,p_r\} \ {\rm and} \  \{\partial_{\bar{z}_1},\cdots,\partial_{\bar{z}_{r-s}},\partial_{t_1},\cdots\partial_{t_s}\} \ {\rm span \ the \ same \ subbundle},\label{B}
\end{equation} 
where $p_j$ is the principal part of $P_j$ for $1\leq j\leq r$ and  $z_i=x_i+\sqrt{-1}y_i$ for $1\leq i\leq r-s$(cf. \cite{N}). Due to the existence of such coordinate charts, it follows from (\ref{eta}) and (\ref{eF}) that $Z$ is a closed submanifold of $\Omega$ with real codimension two. Let $E\subseteq\mathbb{C}T\Omega$ be the subbundle spanned by $\{p_1,\cdots,p_r\}$. Combining (\ref{B}) and the condition $P_1F=\cdots =P_rF=0$, we know that 
\begin{equation}
	E_{{}_Z}:=E\cap\mathbb{C}TZ\label{E}
\end{equation} 
defines a subbundle of  $\mathbb{C}TZ$, i.e., $Z$ is a compatible submanifold of $E$. 

A smooth function on $Z$ is said to be a first integral of $E_{{}_Z}$ if it is annihilated by sections of $E_{{}_Z}$. Taking into account of (\ref{eta}), it is natural to consider the problem of extending a first integral $g_{{}_Z}\in C^\infty(Z)$ of $E_{{}_Z}$ to a function $g\in C^\infty\left(\Omega\right)$ such that $P\left(e^{-\eta}g\right)=0$.  

By using a coordinate chart $ U $ as described in (\ref{lcn.}) and (\ref{B}) which has non-empty intersection with $ Z$, $g_{{}_Z}|_{{}_{Z\cap U}}$ can be extended to a function $ g_{{}_U}\in C^\infty(U)$ such that $P\left(e^{-\eta}g_{{}_U}\right)=0$ holds on $U$. Hence, we can choose an open covering  $ \lbrace\Omega_\alpha\rbrace_{\alpha\in\Lambda} $ of $\Omega$ such that for each $\Omega_\alpha$ intersecting $Z$ there is a smooth function $g_\alpha\in C^\infty(\Omega_\alpha)$ such that $g_\alpha|_{{}_{\Omega_\alpha\cap Z}} = g_{{}_Z}|_{{}_{\Omega_\alpha\cap Z}}$ and $e^{-\eta} g_\alpha\in\Gamma(\Omega_\alpha,\mathcal{S}_P)$. For $\alpha,\beta\in\Lambda$, set  $$ h_{\alpha\beta}:=\frac{g_\alpha-g_\beta}{e^{2\eta} F}\in \Gamma(\Omega_\alpha\cap\Omega_\beta,\mathcal{S}_P) $$ where $g_\alpha:=0$ if $ \Omega_\alpha\cap Z=\emptyset $. It is easy to see that $\{\Omega_\alpha,h_{\alpha\beta}\}_{\alpha,\beta\in\Lambda}$ defines an additive Cousin datum. By (i) of Theorem \ref{Cousin}, one can find some $ h_\alpha\in\Gamma(\Omega_\alpha,\mathcal{S}_P) $ for each $\alpha\in \Lambda$ such that 
$$ h_\beta-h_\alpha=\frac{g_\alpha-g_\beta}{e^{2\eta} F} \ {\rm on} \ \Omega_\alpha\cap\Omega_\beta(\alpha,\beta\in\Lambda).$$
Thus, $ \{g_\alpha+e^{2\eta} h_\alpha F\}_{\alpha\in\Lambda} $ patches together to give an extension $g\in C^\infty\left(\Omega\right)$ of $ g_{{}_Z} $. By the construction of $g$, we know that on each $\Omega_\alpha$
\begin{align*}
	P_j\left(e^{-\eta}g\right)&= P_j\left(e^{-\eta}g_\alpha+e^\eta h_\alpha F\right) \\ &=P_j\left(e^\eta h_\alpha F\right)\\ &= e^\eta FP_j\left(h_\alpha\right) + h_\alpha p_j\left(e^\eta F\right)\\ &= e^\eta h_\alpha P_j\left(F\right)=0,\ 1\leq j\leq r.
\end{align*}
The following proposition records our discussion on the extension of a first integral of $E_{{}_Z}$.
\begin{prop}
Assume that $ \Omega $ is $ P $-convex, $ P=(P_1,\cdots,P_r) $ is elliptic and that $F\in C^\infty\left(\Omega\right)$ is a solution of the homogenous system $(\ref{hom})$ satisfying ${\rm d}F\neq 0$ on $Z:=F^{-1}(0)$. Then for each first integral $g_{{}_Z}\in C^\infty(Z)$ of $E_{{}_Z}$, there exists a function $g\in C^\infty\left(\Omega\right)$ such that $g|_{{}_Z}=g_{{}_Z}$ and $P\left(e^{-\eta} g\right)=0$ on $\Omega$, where $E_{{}_Z}$ is the subbundle of $\mathbb{C}TZ$ defined by $ (\ref{E}) $ and $\eta$ is a solution of $(\ref{eta1})$.
\end{prop}

\subsection*{Acknowledgement}
We would like to thank Prof. F. Treves for his interest in our work. We thank Prof. Paulo D. Cordaro for his valuable suggestions, which enhanced the presentation, especially on the improvement of Corollary \ref{lst}. Finally, we thank the anonymous referee for helpful comments.

\end{document}